\documentclass[11pt]{article}
\usepackage{amssymb, ifthen,amsfonts,amssymb,amsmath,theorem, graphicx}
\usepackage{latexsym}
\usepackage[latin1]{inputenc}
\usepackage{draftwatermark}
        \SetWatermarkText{}
	\SetWatermarkLightness{0.9}
	\SetWatermarkScale{6}
\newcommand{\ds}{\displaystyle}
\newcommand{\RR}{\mathbb R}
\newcommand{\R}{\RR}

\def\H{{\mathcal H}}

\newcommand{\sm}{\setminus}
\newcommand{\g}{\gamma}
\newcommand{\Om}{\Omega}

\newcommand{\vps}{\varepsilon}

\newcommand{\ra}{\rightarrow}

\newcommand{\lra}{\longrightarrow}
\newcommand{\Lra}{\Longrightarrow}

\newcommand{\sr}{\stackrel}

\newcommand{\nif}{{n \rightarrow \infty}}

\newcommand{\sq}{\subseteq}
\newcommand{\ov}{\overline}

 \def\bbbq{{\mathchoice {\setbox0=\hbox{$\displaystyle\rm Q$}\hbox{\raise
 0.15\ht0\hbox to0pt{\kern0.4\wd0\vrule height0.8\ht0\hss}\box0}}
 {\setbox0=\hbox{$\textstyle\rm Q$}\hbox{\raise
 0.15\ht0\hbox to0pt{\kern0.4\wd0\vrule height0.8\ht0\hss}\box0}}
 {\setbox0=\hbox{$\scriptstyle\rm Q$}\hbox{\raise
 0.15\ht0\hbox to0pt{\kern0.4\wd0\vrule height0.7\ht0\hss}\box0}}
 {\setbox0=\hbox{$\scriptscriptstyle\rm Q$}\hbox{\raise
 0.15\ht0\hbox to0pt{\kern0.4\wd0\vrule height0.7\ht0\hss}\box0}}}}

\newtheorem{thm}{Theorem}[section]
\newtheorem{rem}[thm]{Remark}

\newtheorem{lem}[thm]{Lemma}

\newcommand{\qed}{\ifmmode$\Box$\else{\unskip\nobreak\hfil
        \penalty50\hskip1em\null\nobreak\hfil$\Box$
        \parfillskip=0pt\finalhyphendemerits=0\endgraf}\fi}

\newtheorem{demth}{Proof}
        
        \newenvironment{proof}{\begin{demth}\rm}{\qed\end{demth}}

\title{A new isoperimetric inequality for the elasticae }

\author{Dorin Bucur \and Antoine Henrot\thanks{The authors were supported by the Isaac Newton Institute programme "Free Boundary Problems and Related Topics" 2014 and  the ANR Optiform research programme, ANR-12-BS01-0007. }}

\begin{document}
\maketitle
\bigskip

\begin{abstract}
For a smooth curve $\g$, we define its elastic energy as $E(\g)= \frac 12 \int_{\g} k^2 (s) ds$
where $k(s)$ is the curvature.
The main purpose of the  paper is to prove that among all smooth, simply connected,  bounded open sets of prescribed area in $\R^2$, the disc has the boundary with the  least elastic energy. In other words, for any bounded simply connected domain $\Omega$, the following
isoperimetric inequality holds: $E^2(\partial \Om)A(\Om)\ge \pi ^3$.
The analysis relies on the minimization of the elastic energy of drops enclosing a prescribed area, for  which we give as well an analytic answer. \end{abstract}

\section{Introduction}
Let $\Om$ be a smooth, bounded simply connected open set in the plane (the exact smoothness which is required will be made precise in Section \ref{prel}) and let us denote by $\partial \Om$ its boundary. Following L. Euler, we define its elastic energy as
\begin{equation}\label{bhe04e}
E(\partial \Om)= \frac 12 \int_{\partial \Om} k^2 (s) ds
\end{equation}
where $s$ is the curvature abscissa and $k$ is the curvature. We will denote by $A(\Om)$ the area of $\Om$ and $L(\Om)$ its perimeter.  The aim of this paper is to prove the following isoperimetric inequality.
\begin{thm}\label{bhe05}
For any bounded, smooth, simply connected open set $\Om\sq \R^2$
\begin{equation}\label{bhe04}
E^2(\partial \Om)A(\Om)\ge \pi ^3
\end{equation}
where equality holds only for the disc.
\end{thm}
In other words, using the behavior of the elastic energy on rescaling, we get that for every $A_0>0$, the disc is the unique solution for the minimization problem
$$\min\{E(\partial \Om) : A(\Om) \le A_0, \Om\; \mbox{bounded, smooth, simply connected open set of } \R^2\}.$$
More precisely, if we perform any scaling  of ratio $t$, we have $E(t\partial\Omega)=t^{-1} E(\partial \Om)$ and $A(t\partial\Omega)=t^{2} A(\partial \Om)$.
Therefore, it is classical to prove that the  following three minimization problems are equivalent (in the sense that any solution of one gives a solution of the others
after a suitable scaling):
\begin{description}
  \item[(i)] $\min E^2(\partial \Om)A(\Om)$
  \item[(ii)] $\min\{E(\partial \Om) : A(\Om) \le A_0\}$
  \item[(iii)] $\min E(\partial \Om) + A(\Om)$
\end{description}
Let us make some comments. For a detailed bibliography on closed elasticae, we refer to the classical \cite{lang-sing}
or the more recent \cite{Sat2}.
Inequality \eqref{bhe04} was already known for convex domains. Indeed, by a famous inequality due to M. Gage \cite{gage}, for any bounded convex domain
$$\frac{E(\partial \Om) A(\Om)}{L(\Om)} \ge \frac \pi2$$
with equality for the disc. Therefore,
$$E^2(\partial \Om) A(\Om) \ge E^2(\partial \Om) A(\Om) \frac{4\pi A(\Om)}{L^2(\Om)} \ge \frac{\pi^2}{4} \times 4 \pi =\pi ^3,$$
the first inequality being the classical isoperimetric inequality, and the second the Gage inequality. If the convexity hypothesis is dropped, then the Gage inequality is false (as shown by the counter-example of Figure \ref{dumbell}).

The simply connectedness assumption is necessary. Indeed,  if we take as a domain $\Om$ the ring
$$\Om_R =\{ (x,y): R < \sqrt{x^2+y^2} < R+\frac 1R\},$$
we get
$E(\partial \Om_R) = \frac{\pi}{ R}+ \frac{\pi R}{ R^2 +1}$, while
$$A(\Om_R) = \pi (R+\frac 1R)^2 -\pi R^2= 2\pi +\frac{\pi}{R^2}$$
 showing that $E^2(\partial \Om_R)A(\Om_R)\ra 0$ when $R \ra +\infty$.

In the same way, the boundedness assumption is also necessary. Let us consider the following unbounded domain, subgraph of a Gaussian function, but with finite area and elastic energy:
$$\Om_\alpha =\{ (x,y)\in \R^2 : -\infty <x< +\infty, 0<y< e^{-\alpha x^2/2}\}.$$
We have
$$A(\Om_\alpha) =\int_{-\infty}^{+\infty} e^{-\alpha x^2/2} dx =  \sqrt{\frac{2\pi }{{\alpha}}},$$
while
$$E(\partial \Om_\alpha)= \frac 12 \int_{-\infty}^{+\infty} \frac {(\alpha ^2 x^2 -\alpha )^2 e^{-\alpha x^2}}{(1+\alpha ^2 x^2 e^{-\alpha x^2})^\frac{5}{2}}dx =\frac{\alpha ^\frac{3}{2}}{2} \int_{-\infty}^{+\infty} \frac{(u^2-1)^2 e^{-u^2}}{(1+\alpha u^2 e^{-u^2})^\frac{5}{2}} du,$$
and we see that $E^2(\partial \Om_\alpha) A(\Om_\alpha) \ra 0$ as $ \alpha \ra 0$.

This shows that the assumptions in Theorem \ref{bhe05} can not be weakened. The proof of Theorem  \ref{bhe05} is a classical variational proof (existence, regularity, analysis of the optimality  conditions), but the existence part is by no means easy since we need a control on the perimeter of a minimizing sequence. The boundedness constraints on $E(\Om)$ and $A(\Om)$ do not ensure that the perimeter is uniformly bounded, as shown by a counter-example like a dumbell, see Figure \ref{dumbell}.
\begin{figure}[ht]
\begin{center}
\includegraphics[width=8cm]{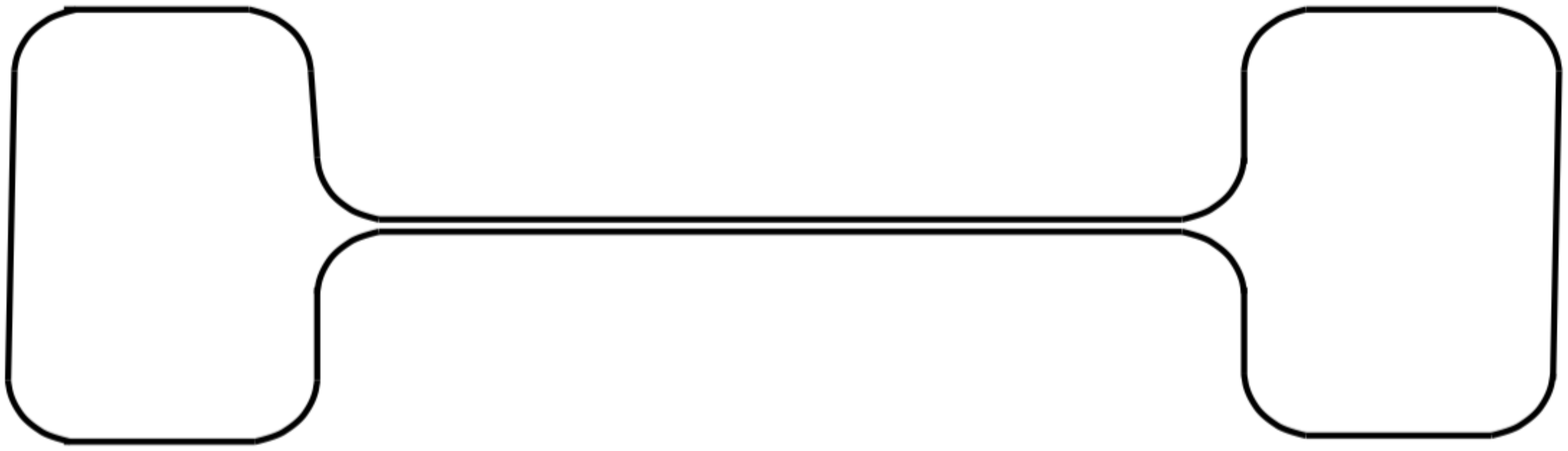}
\end{center}
\caption{A dumbbell with bounded area and elastic energy with a large perimeter}\label{dumbell}
\end{figure}

One has to be particularly careful that a minimizing sequence may a priori have a diameter going to infinity. The key point of our strategy is to analyze first the minimization of the elastic energy of drops enclosing a fixed area, i.e. closed loops without self-intersection points, which are smooth except one point, where the tangents are opposite. The result for drops will straight forward imply the conclusion of Theorem \ref{bhe05}, since we prove that the optimal loop is smooth  and satisfies  optimality conditions over the full boundary. A direct observation proves that the curve is a circle.

Here is our plan.
\begin{itemize}
\item We  solve the minimization problem
\begin{equation}\label{bhe06}
\min\{ E(\partial \Om)+ A(\Om) : \Om \; \mbox{ open, smooth, bounded, simply connected}\},
\end{equation}
which is equivalent to  \eqref{bhe04}.
\item We fix some radius $R>0$ and replace problem \eqref{bhe06} by
\begin{equation}\label{bhe07.b}
\min\{ E(\partial \Om)+ A(\Om) : \Om \sq B_R \; \mbox{ open, smooth,  simply connected}\},
\end{equation}
where $B_R$ is the ball centered at $0$ of radius $R$. We prove that every simply connected domain in $B_R$ satisfies
$L(\Om) \le R^2 E(\partial \Om)$. This is a key point for proving existence.
\item In order to be able to exploit the optimality conditions we have to deal with self-intersection points and with the points where the optimal set is touching the boundary of the ball $B_R$. For this reason, we analyze the problem
\begin{equation}\label{bhe08.b}
\min\{ E(\partial \Om)+ A(\Om) : \Om \sq B_R \; \mbox{ open, smooth, simply connected {\it drop}}\}.
\end{equation}
We refer to Section \ref{drops} for a precise definition of drops. We prove that  an optimal drop does not have self-intersection points and, if $R$ is large enough, it does not touch the boundary of the ball $B_R$ (up to a translation, inside the ball). Henceforth, optimality conditions allow us to exhibit precisely the optimal drop  and to evaluate its energy. This drop turns out to be {\it unique}.
\item We come back to problem \eqref{bhe07.b} and prove that a limit of minimizing sequence can not have self-intersection points and can not touch the boundary of $B_R$, provided that $R$ is large enough. Consequently, optimality conditions can be  written on all its boundary.     The elimination of self-intersection points relies on the previous result on drops, since the presence of at least one such point would make the energy not smaller than the double of the energy of the optimal drop, which turns out to be larger than the energy of a disc of radius $2^{-1/3}$. As optimality conditions can be written on the full loop, we prove that there are only
four shapes (all of them having a very simple description) which can be minimizers. An easy comparison argument, leads to the optimality of the ball.
\end{itemize}

\section{Preliminaries}\label{prel}

All curves $\g:[0,L]\ra \R^2$ are parametrized by the arc-length. We denote $\theta$ the angle of the tangent to $\g$ with respect to the axis $Ox$.
The curvature of $\g$ at the point $\g(s)$ will be denoted $k(s)$ and it is equal to $\theta '(s)$.  Since we shall work with curves with finite elastic energy, the function $\theta$  belongs to the Sobolev space $H^1(0,L)$. Using the embedding $H^1(0,L) \sq C^{0, \alpha} [0,L]$, for any $\alpha <1/2$, the function  $\theta$ is, in particular,  continuous.

All curves we work in this paper have finite elastic energy
$$E(\g)=\frac 12 \int_{[0,L]} |\theta '(s)|^2 ds<+\infty .$$

\begin{lem}\label{bh04}
Let $M>0$ and $\g :[0,L] \ra \R^2$ be a curve parametrized by the arc length such that $E(\g) \le M$. There exists $l=l(M) >0$ such that for every $s_0 \in [0, L-l]$ the curve is a graph in a local system of coordinates with a first axis aligned on $\theta (s_0)$, on a segment $[0, \frac{l}{\sqrt{2}} ]$, of a function $g : [0, \frac{l}{\sqrt{2}}] \ra \R$ with $g(0)= \g(s_0)$,   $g'(0)=0$ and such that $\forall t \in (0,L)\; |g'(t)| \le 1$.
\end{lem}
\begin{proof}
Let us fix $s_0$ and consider the smallest $s_1 >s_0$ such that $|\theta (s_1)-\theta(s_2)|=\frac\pi 4$. If $s_1$ does not exist, then the conclusion follows directly.

We reproduce the curve $\g_{[s_0,s_1]}$ eight times, taking successively a reflection with respect to the  line passing trough the point $\g(s_1)$ and orthogonal to the tangent at $\g(s_1)$, then the same procedure for the image of $\g(s_0)$ and a last reflection in order to close the loop.

\begin{figure}[ht]
\begin{center}
\includegraphics[width=6cm]{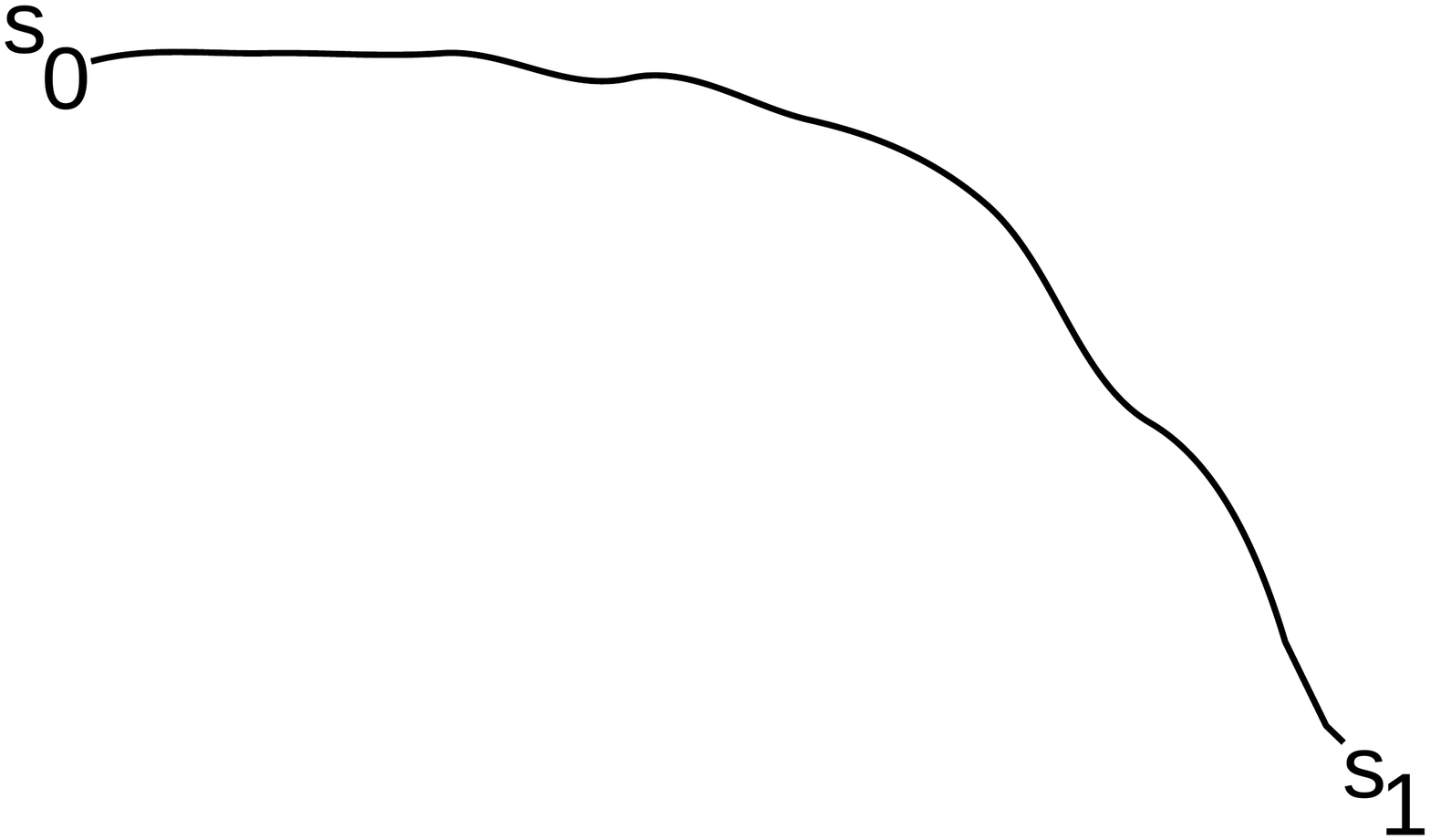}\includegraphics[width=6cm]{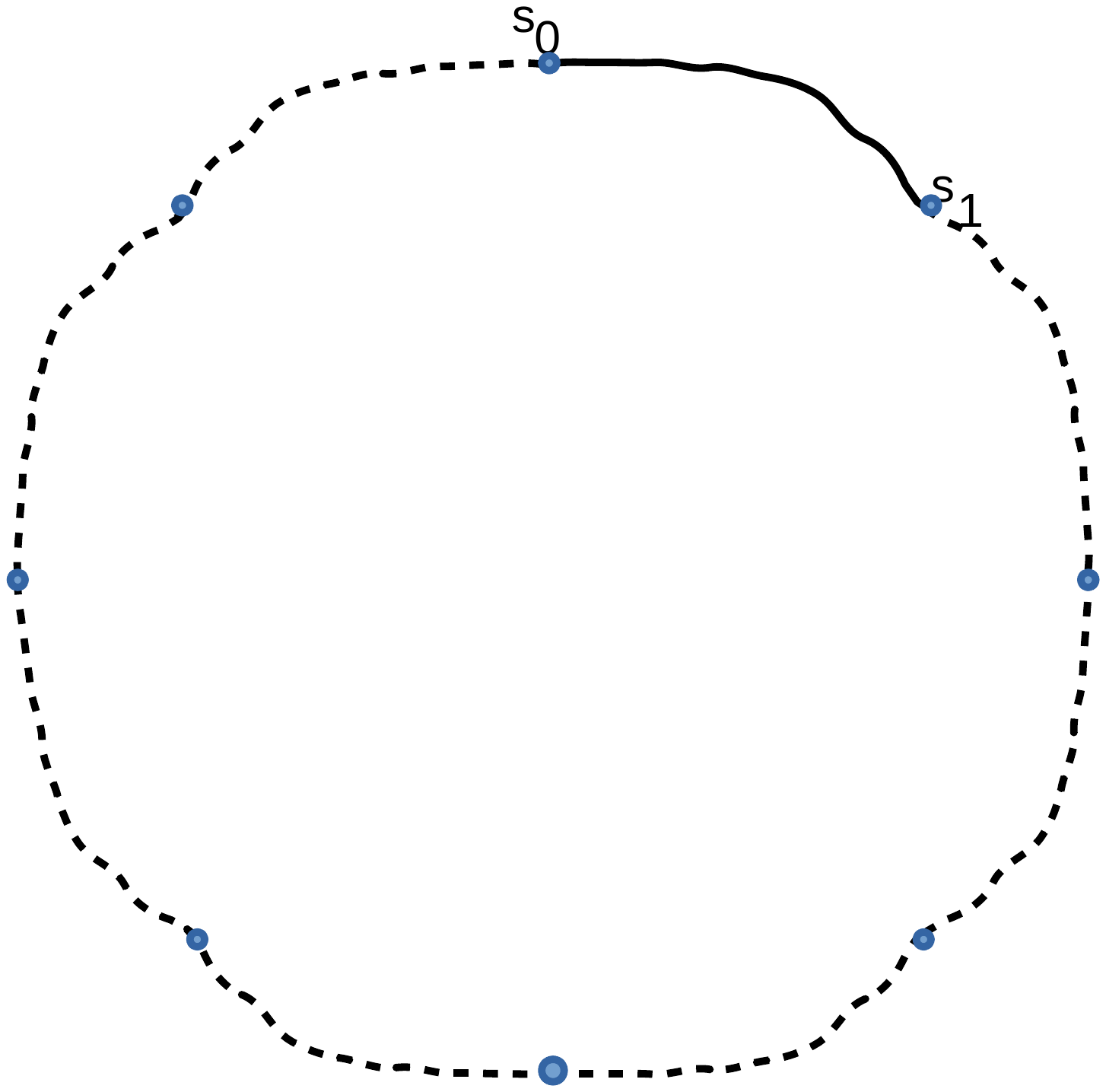}
\end{center}
\caption{Initial curve $\g|_{[s_0,s_1]}$ and the (rescaled) loop built from the curve}\label{fig1n}
\end{figure}
Let us denote by $C$ the curve which is the boundary of the convex envelope of the loop. From \cite{gage}, we have
$$\int_C k^2 ds  \ge \frac{\pi |\H^1(C)|}{Area(C)} \ge \frac{\pi |\H^1(C)|}{\frac{|\H^1(C)|^2}{4\pi}} \ge \frac{4\pi^2}{ 8(s_1-s_0)}.$$
But
$$8\int_{\g_{[s_0,s_1]}} |\theta'|^2 ds \ge \int_C k^2 ds $$
thus $$128 M \ge \frac{4\pi^2}{ s_1-s_0},$$
hence
$$s_1-s_0\ge \frac{\pi^2}{32M}.$$
Denoting $l= \frac{\pi^2}{32M}$, we conclude the proof.

\end{proof}
\begin{rem} {\rm
The assertion of this lemma is of course available on backwards, so that the curve is locally a graph in a neighborhood of each point, over an interval of controlled length.
}
\end{rem}
\begin{lem}\label{bh02}
Let $\g :[0,L] \ra \R^2$ be a curve parameterized by the arc length such that $E(\g)< +\infty$. If $\vps >0$ is given and $0\le s<t\le L$ are such that
$$|\theta (s)-\theta(t)|=\vps,$$
then
$$\int_{[s,t]} |\theta'|^2ds  \ge \frac{\vps^2}{L}.$$
\end{lem}
\begin{proof}
As $\int _{[0,L]} |\theta'|^2ds <+\infty $, we write
$$|\theta (s)-\theta(t)| =\left|\int_s^t \theta'(u) du \right| \le |t-s|^\frac12 (\int_{[s,t]} |\theta'|^2 )^\frac12$$
which gives the result.
\end{proof}
\begin{rem} {\rm
The idea coming out of the lemma is that if there is an $\vps$-variation of the angle, the elastic energy on that section of the curve is at least a constant times $\vps^2$, the constant depending on the global  length of the curve. }
\end{rem}
Let $B_R$ be a ball of radius $R$.
\begin{lem}\label{bhe20}
Let $\g :[0,L] \ra \R^2$ be a smooth loop parameterized by the arc length such that $E(\g)< +\infty$ and $\g([0,L]) \sq B_R$. Then
$$L\le 2R^2 E(\g).$$
\end{lem}
\begin{proof}
Denoting $\g(s)=(x(s), y(s))$, we have
$$L= \int _0^L x'^2(s) +y'^2(s) = - \int _0^L x(s)x''(s)+ y(s) y''(s) ds.$$
But $ | x(s)x''(s)+ y(s) y''(s)| \le (x^2(s)+y^2(s))^{\frac 12} (x''^2(s)+y''^2(s))^{\frac 12} \le R |k(s)|$. Therefore, the conclusion of the lemma follows from the Cauchy-Schwarz inequality
$$L^2 \le R^2 L \int _0^L k^2(s) ds.$$
\end{proof}

Assume that a simply connected open set $\Om$ is bounded by  a loop $\g$ which does not have self intersections on $(s_0,s_0+L)$. We shall call this piece of curve, a free branch. Let us first give the optimality conditions satisfied by any free branch of an optimal domain. On such a free branch, we can perform any (small and compactly supported, smooth) perturbation.
\begin{thm}[Optimality conditions]\label{opti}
Let $\g$ be any free branch of a minimizer $\Om$ of  the energy $E(\partial\Om)+A(\Om)$. Then $s\mapsto k(s)$ is $C^\infty$
on $\g$ and satisfies:
\begin{enumerate}
\item [(B1)] $k''=-\frac 12 k^3  +1$
\item [(B2)] $k'^2=-\frac 14 k^4 +2k+ 2C$, for some constant $C$
\item [(B3)] $\exists Q \in \R^2$, such that $\forall M\in \g$:  $QM^2= 2k+2C$, for some constant $C$
\item [(B4)] $\exists Q \in \R^2$, such that $\forall M\in \g$: $QM.\nu= \frac 12 k^2$ where $\nu$ is a normal vector to $\g$.
\end{enumerate}
\end{thm}
\begin{rem}
The point $Q$ in (B3), (B4) is the same (see the proof below). The constant $C$ in (B2), (B3) is also the same. To see that, take
a point $M_M$ on $\g$ where the curvature $k$ is maximum. If this point does not exist, just extend the curve
with the same ODE. Then, according to (B3), $QM_M$ is also maximum and the normal derivative of the boundary at this point is $QM_M/|QM_M|$.
Therefore (B4) yields $QM_M=\frac 12 k^2$ and plugging into (B3) gives (B2), because $k'=0$ at this point, with the same constant.
\end{rem}
\begin{proof}
The $C^\infty$ regularity of $k(s)$ (and $\theta(s)$) comes from a bootstrap argument and equation (\ref{opti2}) below.
The first condition (B1) comes from the classical {\it shape derivative} of the elastic energy (under small perturbation
of the boundary driven by some smooth vector field $V:\R^2 \to \R^2$), see \cite[chapter 5]{HP} for more details on the shape derivative. Following e.g. the Appendix in
\cite{bht14}, we see that it is given by
$$d E(\partial \Om,V)=-\int_\g  \left( \frac 12 k(s)^3 + k''(s) \right) \langle V, \nu\rangle \ ds$$
while the derivative of the area is classically
$$d A(\Om,V)=\int_\g  \langle V, \nu\rangle \ ds$$
Condition (B1) follows since the derivative of $E+A$ must vanish for any $V$. We obtain condition (B2) multiplying
(B1) by $k'$ and integrating.

To get condition (B3), we use another expression of the elastic energy and the area. Namely, with the parametrization with the angle $\theta$
we have (see \cite{bht14} for more details):
$$E(\g)=\frac 12 \int_\g {\theta'}^2 ds:=e(\theta),\quad A(\Om)=\int\int_T \cos\theta(u)\sin\theta(s)\,du\,ds:=a(\theta)$$
where $T$ is the triangle $T=\{(u,s)\in \mathbb{R}^2 \ ; \ 0\leq u\leq s\leq L(\Omega)\}$. We note $L$ for $L(\Om)$. Thus we are led to minimize the sum
$e(\theta)+a(\theta)$ with the following constraints (the starting and the ending point of the branch $\g$ are fixed).
\begin{equation}\label{opti0}
\int_0^L \cos(\theta(s))\,ds=x(L)-x(0), \qquad \int_0^L \sin(\theta(s))\,ds=y(L)-y(0).
\end{equation}
The derivative of $e(\theta)$ is (for a perturbation $v$ compactly supported)
$$\langle de(\theta),v\rangle=\int_0^L \theta' v' ds= - \int_0^L \theta'' v ds$$
while the derivative of $a(\theta)$ is given by
$$\langle da(\theta),v\rangle=\int\int_T \cos\theta(u)\cos\theta(s) v(s) - \sin\theta(s)\sin\theta(u) v(u) du ds.$$
Using (\ref{opti0}) and Fubini, we can write
$$\int\int_T \sin\theta(s)\sin\theta(u) v(u) du ds= (y(L)-y(0))\int_0^L \sin\theta(s)v(s) ds
- \int\int_T \sin\theta(u)\sin\theta(s) v(s) du ds.$$
Therefore, the optimality condition for the constrained problem reads: there exists Lagrange multipliers $\lambda_1,\lambda_2$
such that, for any $v$:
\begin{eqnarray}\label{opti1}
    - \int_0^L \theta'' v ds + \int_0^L \left(\cos\theta(s) \int_0^s \cos(\theta(u) du + \sin\theta(s) \int_0^s \sin(\theta(u) du\right) v(s) ds =\\ \nonumber
 = (y(L)-y(0))\int_0^L \sin\theta(s)v(s) ds  - \lambda_1 \int_0^L \sin\theta(s) v(s) ds + \lambda_2 \int_0^L \cos\theta(s) v(s) ds.
\end{eqnarray}
which implies (thanks to $x'(s)=\cos\theta(s), y'(s)=\sin\theta(s)$)
\begin{equation}\label{opti2}
-\theta''+ x'(x-x(0))+y'(y-y(0))=(y(L)-y(0) -\lambda_1) y' + \lambda_2 x'
\end{equation}
By integration, we get (B3) setting $Q=(x(0)+\lambda_2, y(L)-\lambda_1)$.

At last, differentiating twice (B3) we get $k'=QM.\tau$ (where $\tau$ is the tangent vector) and
$k''=1- k QM.\nu$. Using (B1) we see that $\frac 12 k^3 = k QM.\nu$, so (B4) holds where $k\not= 0$. Since $k$ is a solution
of the ODE (B1), and therefore can be written with elliptic functions, it can only vanish on isolated points and thus (B4) holds everywhere by continuity of both members.
\end{proof}
In the following lemma, we  assume that the simply connected open set $\Om$ is a minimizer of  the energy $E(\partial\Om)+A(\Om)$.
\begin{lem}\label{be03}
Any free branch of a  minimizer $\Om$ has a length $L$ uniformly bounded by
$$L\le 146.$$
\end{lem}
\begin{proof}
We work with a free branch of $\g$ on $s \in (s_0,s_0+L)$ and use the  optimality conditions above.
We also know that the elastic energy of this branch is less than the total energy of the best disk $B$, so that
\begin{equation}\label{B4}
E(\g) \le E(\partial B)+A(B) =  3\pi 2^{-\frac 23}.
\end{equation}

We consider two cases. Assume first that $C \le 1$ on this branch ($C$ is defined above in (B2), (B3)). Then we know from (ODE3) in the Appendix that
\begin{equation}\label{seeODE3}
k(s) \le k_M(C) \le k_M(1) \le \frac 73
\end{equation}
Then, from (B3)
$$QM^2\le \frac{14}{3} +2= \frac{20}{3},$$
hence the arc is contained in the disc centered at $Q$ with radius $R_0=\sqrt{\frac {20}{3}}$.

On the other hand, if we put the origin at $Q$
$$L(\g)=L=\int_0^L x'^2+y'^2 dx = (xx'+yy')|_0^L - \int_0^L xx''+yy'' ds.$$
But $|x(L)x'(L)+y(L)y'(L)| \le R_0$ and $|x(0)x'(0)+y(0)y'(0)| \le R_0$ while by Cauchy-Schwarz and \eqref{B4}
$$|\int_0^L xx''+yy'' ds |\le R_0 \int_0^L |k| ds \le R_0 \sqrt{L2E(\g)}\le R_0 \sqrt{L 3\pi 2^{\frac 13}}.$$
Therefore, $L$ satisfies
\begin{equation}\label{B5}
L\le 2 \sqrt{\frac{20}{3}}+ \sqrt {\frac{20}{3}\times 3\pi\times 2^{\frac 13}} \sqrt{L},
\end{equation}
which implies (as soon as $C \le 1$)
\begin{equation}\label{B5b}
L\le 90.
\end{equation}

\medskip
Second case: $C\ge 1$ for this branch. In this case we have from (ODE3) in the Appendix
$$k_M(C)\ge k_M(1)\ge \frac 94,$$
$$k_m (C)\le k_m(1) \le - \frac {9}{10}.$$

We decompose the interval $I= (s_0, s_0+L)$ in three parts (some could
be empty), $I=I_-\cup I_0\cup I_+$ where
$$I_-=\{s \in I : k(s)\le 0 \} $$
$$I_0=\{s \in I : 0< k(s) < 2^{\frac 13} \} $$
$$I_+=\{s \in I : 2^{\frac 13}\le k(s) \} $$
and we are going to prove that the length of each part is uniformly bounded, by a controlled constant. First of all, we have seen that the integral of $k^2$ on a period satisfies (see ({ODE4}) in the Appendix)
$$\frac 12 \int_0^T k^2 ds\ge \frac{\pi}{4} \sqrt{\frac{22}{3}}.$$
Following \eqref{B4}, this implies that we can not have more than $3$ periods on each free branch.
We begin with  $I_+$. Obviously
$$E(\g) \ge \frac 12 \int_{I_+} k^2 ds \ge \frac 12 2^{\frac 23} |I_+|,$$
therefore
\begin{equation}\label{B7}
|I_+|\le 3\pi \times 2^{-\frac 23}\times 2^{\frac 13}\le 8.
\end{equation}
For $I_0$, we consider one of its connected components, say $(\alpha, \beta)$. Since $k_M(C)\ge \frac 94> 2^{\frac 13}$ and $k_m(C) \le - \frac {9}{10} <0$, we cannot have any local minimum or local maximum of $k$ in $I_0$ according to (ODE2) form the Appendix. Therefore, $k$ is either increasing from $k(\alpha)$ to $k(\beta)$, or decreasing from $k(\alpha)$ to $k(\beta)$. Moreover, there are at most $6$ such connected components because there are at most $3$ periods of $k$. Let us consider the case of $k$ increasing from $k(\alpha)$ to $k(\beta)$, the other one being similar. We have $0 \le k(\alpha)\le k(\beta) \le 2^{\frac 13}$. By (B1), $k'' \ge 0$ on $(\alpha, \beta)$, so that $k$ is convex. Therefore
\begin{equation}\label{B8}
k(\alpha)+k'(\alpha) (s-\alpha) \le k(s)
\end{equation}
which implies
$$k(\alpha)+k'(\alpha) (\beta-\alpha)\le k(\beta) \le 2^{\frac 13}.$$
Now $k(\alpha ) \ge 0$ and $k'(\alpha) = \sqrt{ 2C+2k(\alpha)-\frac 14 k^4(\alpha)  }\ge \sqrt{  2C }$ thus
$ \sqrt{2} (\beta-\alpha) \le \sqrt{  2C }  (\beta-\alpha) \le 2^{\frac 13}$ or $\beta-\alpha \le 2^{-\frac 16}$.

Since, there are at most $6$ such intervals, we have
\begin{equation}\label{B9}
|I_0| \le 6 \times 2^{-\frac 16} \le 6.
\end{equation}
At last we consider the case of $I_-$. The set $I_-$ is not empty only when $C>0$ and $k_m <0$. The set $I_-$ is  composed of connected components $[\alpha, \beta]$ such that $k(\alpha)=k(\beta)=0$ or is included in such connected components. Since we want to estimate from  above the length of $I_-$, it suffices to look for
the length of these connected components. There are at most $3$ of these (identical) components and $k(\frac{\alpha+\beta}{2})=k_m$ by symmetry.

Now, the elastic energy of such a component satisfies
\begin{equation}\label{B10}
E(\g_{\alpha_1, \beta_1}) = \frac 12 \int _{\alpha_1}^{\beta_1}k^2 ds = \int _{\frac{\alpha_1+\beta_1}{2}}^{\beta_1} k^2 ds = \int ^{\frac{\alpha_1+\beta_1}{2}}_{\alpha_1} k^2 ds.
\end{equation}
We denote $L_-= \beta_1-\alpha_1$ the length of this component. By convexity, on $(\alpha_1, \frac{\alpha_1+\beta_1}{2})$ we have
$$k(s) \le \frac{2k_m}{L_-}
(s-\alpha_1)\le 0,$$
thus
$$ E(\g_{\alpha_1, \beta_1})\ge \int_{\alpha_1}^{\alpha_1 + \frac{L_-}{2}}
\frac{4k_m^2}{L_-^2} (s-\alpha_1)^2 ds = \frac{k_m^2}{6}L_-.$$
Now, for $C\ge 1$ we have (see (ODE3) in the Appendix) $k_m^2 \ge k_m^2(1)\ge \frac{81}{100}$     and
$E(\g_{\alpha_1, \beta_1}) \le 3\pi 2^{-\frac{2}{3}}$. Therefore
$$ L_- \le \frac{600}{81} \times 3 \pi 2^{-\frac 23}\le 44,$$
and the total length of
\begin{equation}\label{B11}
|I_-|\le 3 L_-\le 132.
\end{equation}
In conclusion, for $C\ge 1$ the total length is less than (by gathering \eqref{B7}, \eqref{B9}, \eqref{B11})
$$L\le 132+8+6=146.$$
\end{proof}

\section{The optimal drop}\label{drops}

In this section we prove the existence of a best {\it drop} minimizing the sum of the elastic energy and the area enclosed. We introduce the class of  admissible {\it Jordan drops} consisting of simply connected open sets $\Om$ bounded by a Jordan curve $\g$ of finite length, which satisfies
$$ \theta (0)=\theta (L_\g)-\pi, \;\; E(\g)< +\infty, \;\;$$
where $L_\g$ is the length of $\g$.
\begin{figure}[ht]
\begin{center}
\includegraphics[width=6cm]{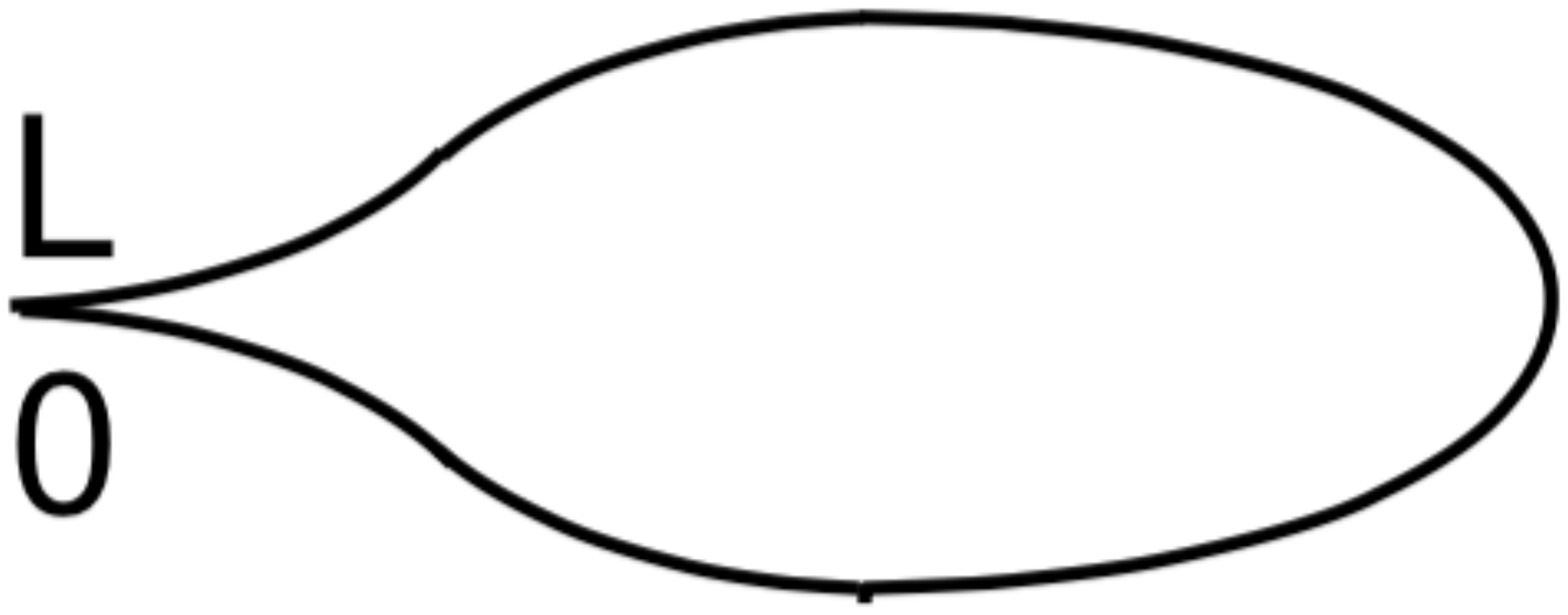}
\end{center}
\caption{A drop}\label{fig1nb}
\end{figure}
A drop will be denoted $(\Omega, \g)$, $\Omega$ being the open set enclosed by the Jordan curve $\g$ (all Jordan curves are oriented in the positive sense).

 For some $R>0$, we consider the problem
\begin{equation}\label{bh01}
\inf\{E(\g) +A(\Om) : (\Om,\g) \mbox{ is a drop}, \Om \sq B_R\}.
\end{equation}

Note that by a similar argument as in Lemma \ref{bhe20}, the length of Jordan drop $\g$ can not exceed   $8 R^2 E(\g)$. Indeed, the same argument works for the drop, if the singularity lies at the origin, we have $x^2+y^2 \le 4R^2$ since the diameter of the drop is less than $2R$.

Here is the main result.
\begin{thm}\label{bh03}
Problem \eqref{bh01} has at least one solution.
\end{thm}
\begin{rem} {\rm
With no assumptions on the radius $R$, it could be possible that the optimal drop $(\Om,\g)$ touches the boundary of the ball but it may not have self intersections.
}
\end{rem}
\begin{proof}
For simplicity of the notation, the ball $B_R$ will be denoted $B$ and the area of $\Om$ will be denoted by $|\Om|$. We start with the following.
\begin{lem}
Let $(\Om,\g)$ be a drop contained in $B$. If for some $\vps >0$ there exists $0\le s<t \le L_\g$ with
$$\theta(t)=\theta (s)-\pi-\vps$$
then there exists a new drop $(\tilde \Om, \tilde \g)$ in $B$ such that
$$\int_{\tilde \g} |\tilde \theta '|^2 \le \int_{ \g} |\theta '|^2 -\frac{\vps ^2}{2L_\g} \;\;and \;\;|\tilde \Om|\le |\Om|.$$
\end{lem}
 \noindent [Proof of the Lemma] Assume $s$ and $t$ satisfy the hypotheses. Then, from continuity of $\theta$, there exists
$s<\ov s<\ov t < t$ such that
$$\theta (\ov s)=\theta (s) -\frac{\vps}{2} \;\;and \;\;\theta (\ov t)=\theta (t) +\frac{\vps}{2}.$$
Moreover, there exists
$\ov s\le s'<t'\le \ov t$ such that
$$\theta (t')=\theta (\ov t), \theta (s') =\theta (\ov s)$$ and for every $u \in(s',t')$
$$\theta (u) \in (\theta (t'),\theta (s')).$$
Indeed, we define
$$t'=\inf \{ t : t>\ov s, \theta (t)=\theta (\ov t)\},$$
and
$$s'=\sup \{ s : s < t', \theta (s)=\theta (\ov s)\}.$$
Then we notice that the curve $\g_{[s',t']}$ is a graph in the direction $\theta (s')$, otherwise it would contradict the choice of $s'$ and $t'$.
\begin{figure}[ht]
\begin{center}
\includegraphics[width=8cm]{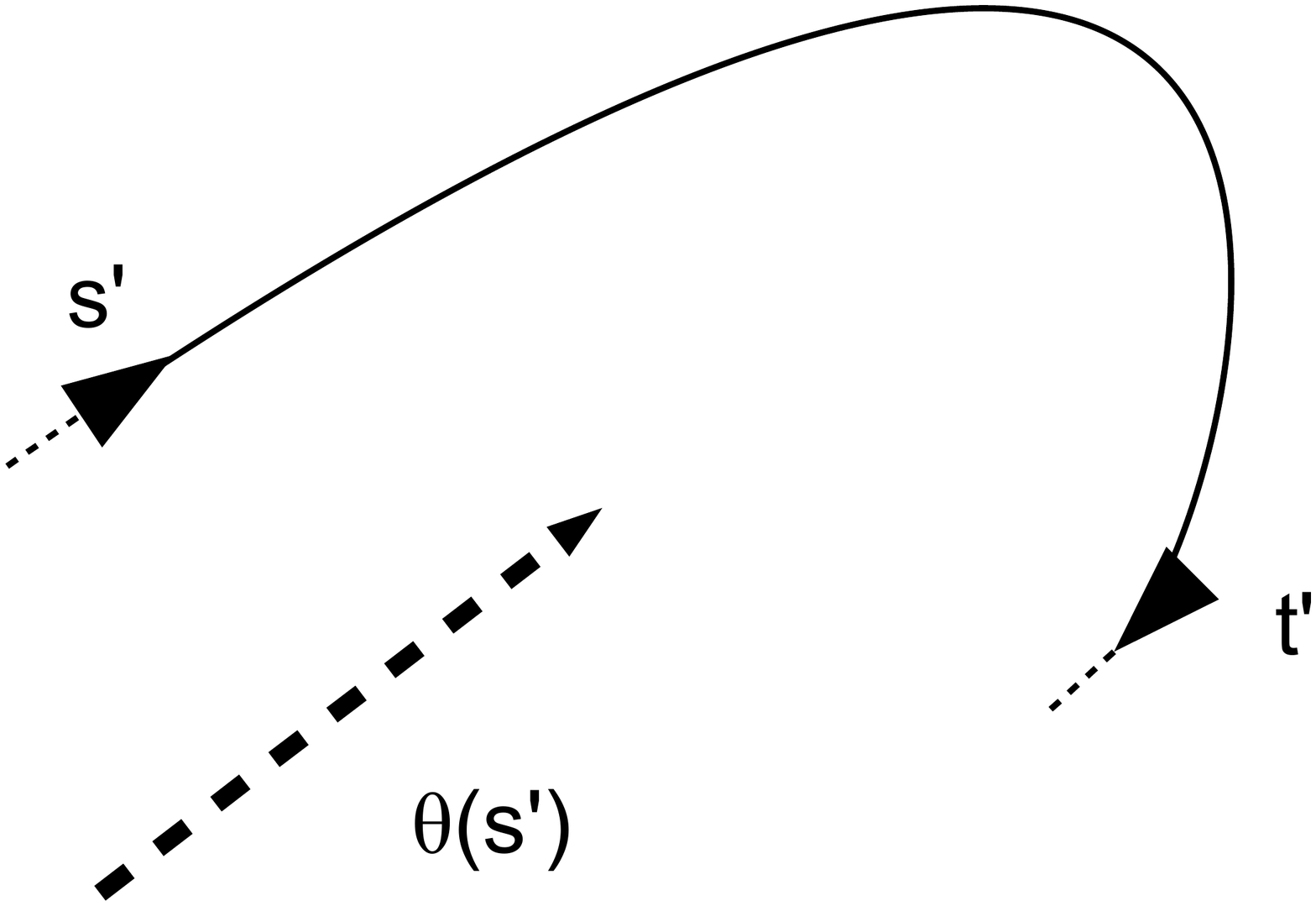}
\end{center}
\caption{The curve is a graph in the direction $\theta (s')$}\label{fig1na}
\end{figure}
Setting the orientation of the curve in the trigonometric sense, we are in  configuration similar to Figure \ref{bh07}.
\begin{figure}[ht]
\begin{center}
\includegraphics[width=8cm]{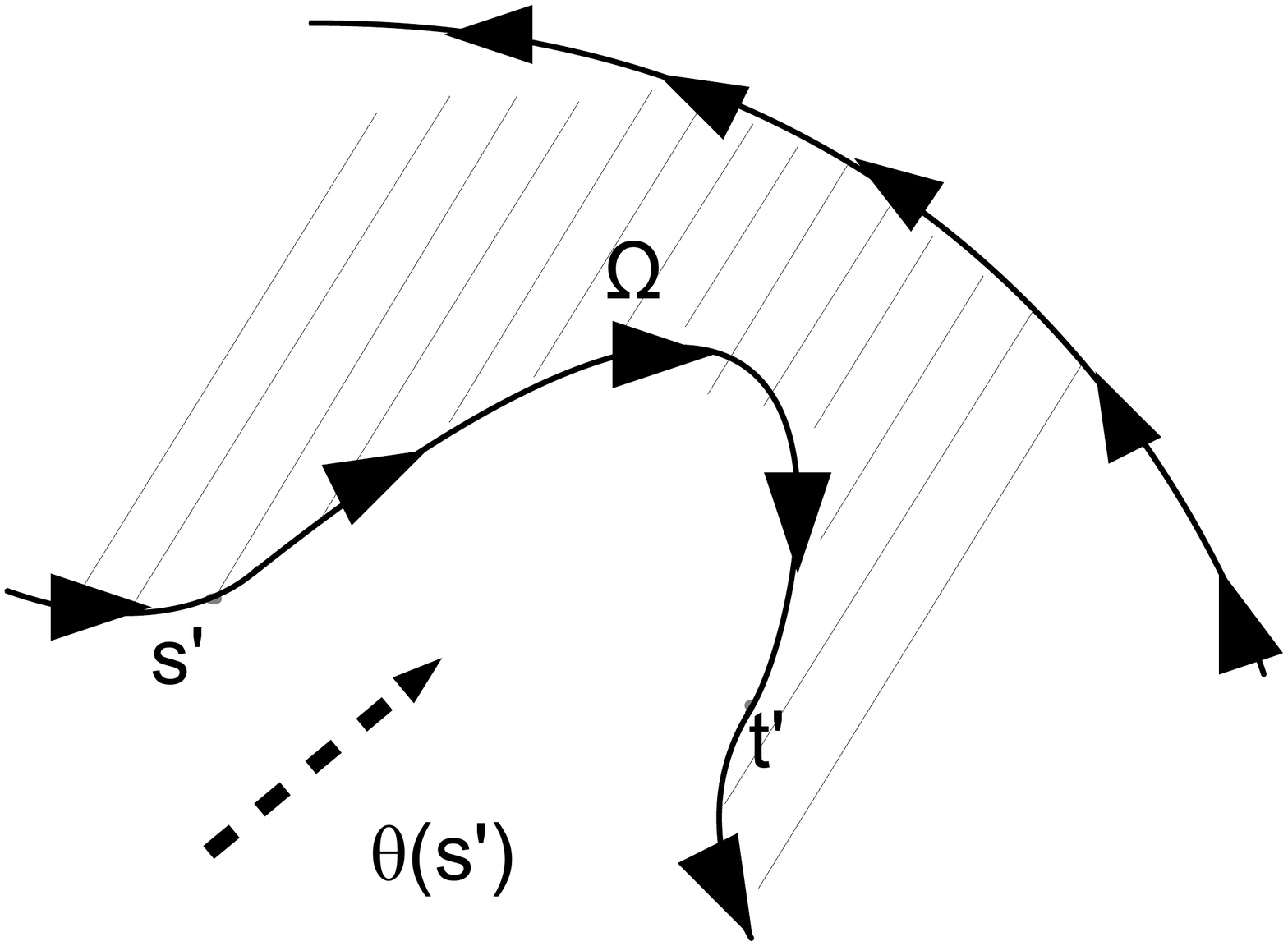}
\includegraphics[width=8cm]{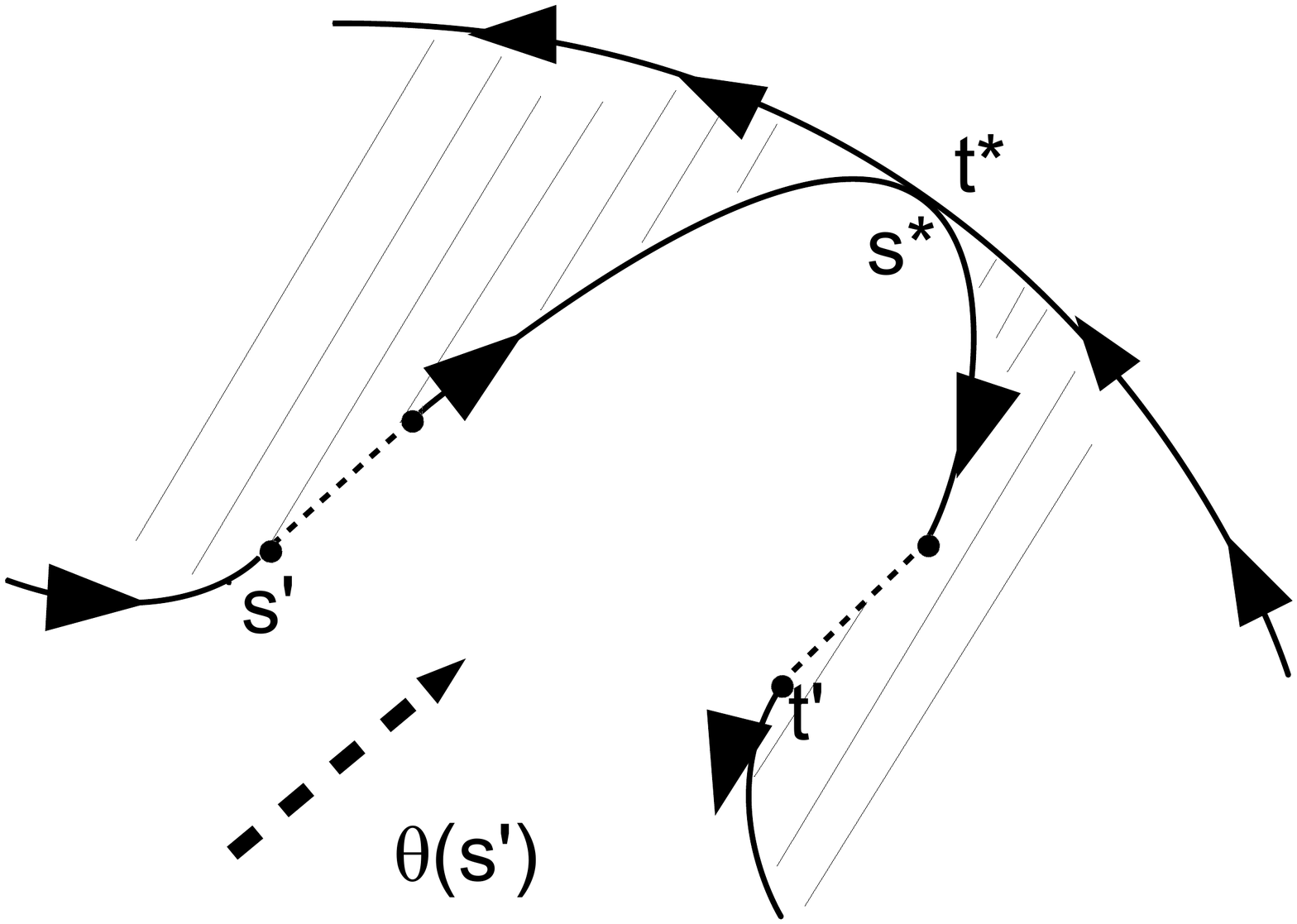}
\end{center}
\caption{Translation of  $\g|_{[s',t']}$ in the direction $\theta (s')$}\label{bh07}
\end{figure}
Using the graph property, we can translate continuously the piece of the curve $\g|_{[s',t']}$ in  a parallel way in the direction $\theta (s')$ until this piece touches again $\g$.

We denote $s_\alpha \in [s',t']$ and $t_\alpha\in [0,L]\sm [s',t']$ the couples of touching points. We denote $s_1$, respectively $s_2$, the minimal and maximal values of $s_\alpha$. Then, one of the curves starting with $s_2$ and ending in $t_2$, or starting in $t_1$ and ending in $s_1$ is a drop. Precisely, it is the one which does not contain the point $\g(0)$. Without loosing generality we can assume it is the curve $s_2\rightarrow t_2$ and rename the point $(s_2,t_2)=(s^*,t^*)$ and denote this curve $\tilde \g$. We notice that $\tilde g$ can not touch any the piece of curve $\g|_{[\ov s, s']}$. If there would be a contact point, this contact is generated by the translation of $\g|_{[ s', t']}$ and has to be precisely $(s^*, t^*)$. But in this case, $t^*$ lies in the interval $[\ov t, s']$, so the curve starting at $t^*$ and ending at $s^*$ is a drop, which does not touch the piece of curve $\g|_{[t', \ov t]}$.

In this way, we built a new drop $(\tilde \Om, \tilde \g)$, which encloses a domain contained in $\Om$ and, in view of Lemma \ref{bh02} has an elastic energy smaller by at least an increment of $\frac{\vps^2}{4L_\g}$.

\noindent [Proof of the Theorem \ref{bh03} (continuation)] Coming back to the proof of Theorem \ref{bh03}, les us consider $(\Om_n, \g_n)$ be a minimizing sequence of drops. We may assume that
$E(\g_n)$, $|\Om_n|$  and $L_{\g_n}$ are convergent. Assume that for every $n$ we have $L_{\g_n} \le L^*$. In order to work on a fixed Sobolev space  $H^1(0,L^*)$, we assume that $\theta_n$ is formally extended by the constant $\theta _n (L_{\g_n})$ on $(L_{\g_n}, L^*]$. Up to a subsequence, we can assume that $\theta_n$ converges uniformly on $[0,L^*]$ to some function $\theta$. We define the limit curve $\g$ in the following way: $L_\g =\lim_\nif L_{\g_n}$ and
$\g: [0, L_\g]\ra \R^2$, $\g(s) =\int_0^s e^{i\theta(s)} ds + a$, where $a =\lim_\nif \g_n(0)$.

 Let us fix $\vps >0$. Then, from the previous lemma, for every $s<t$ and $n$ large enough we have
$$\theta_n(t)\ge \theta_n(s)-\pi-\vps.$$
 Indeed, otherwise we would replace $(\Om_n,\g_n)$ by $(\tilde \Om_n, \tilde g_n)$ decreasing the energy by a fixed increment  $\frac{\vps^2}{4L^*}$, where $L^*$ is a bound of the lengths. This is in contradiction with the minimality of the sequence.

 In particular, passing to the limit we get that for every $\vps >0$ and  for every $s<t$
$$\theta(t)\ge \theta(s)-\pi-\vps.$$
Since $\vps$ is arbitrary, we get
\begin{equation}\label{bh06}
\theta(t)\ge \theta(s)-\pi.
\end{equation}
From the compactness of the class of closed subsets of $\ov B_R$ endowed with the Hausdorff metric, and the embedding of $H^1(0,L^*)$ into $C^{0,\alpha} [0,L^*]$, we may assume that for some open set $\Om \sq B_R$
 $$\Om_n^c \sr{H}{\lra} \Om^c,$$
 and the convergence of $\theta_n$ leads to
 $$\g_n([0,L_{\g_n}]) \sr{H}{\lra} \g([0,L_\g]).$$
 We refer to \cite{bubu05} or \cite{HP} for precise properties of the Hausdorff convergence.
 We know that in general $\ds 1_\Om\le \liminf_\nif 1_{\Om_n}$, so that $\ds |\Om| \le \lim_\nif |\Om_n|$. Nevertheless, in our situation the perimeters being uniformly bounded, we get $1_{\Om_n}\ra 1_\Om$ in $L^1(B_R)$. Moreover, $\partial \Om\sq \g ([0, L_\g])$ and  $\Om$ is simply connected (i.e. any loop contained in $\Om$ is homotopic to a point in $\Om$), but not necessarily connected. The curve $\g$ is possibly self-intersecting, but not crossing, i.e.  at every self-interesting point, the tangent line is the same, while looking locally around the point, the pieces of curve passing through it are  (in view of Lemma \ref{bh04}) graphs of functions. From the simple connectedness hypothesis, these functions are necessarily ordered. From Lemma \ref{bh02} and the fact that the elastic energy is finite, the number of pieces of curve passing through the touching point is uniformly finite.
The situation displayed in Figure \ref{fig6c} may occur.
\begin{figure}[ht]
\begin{center}
\includegraphics[width=8cm]{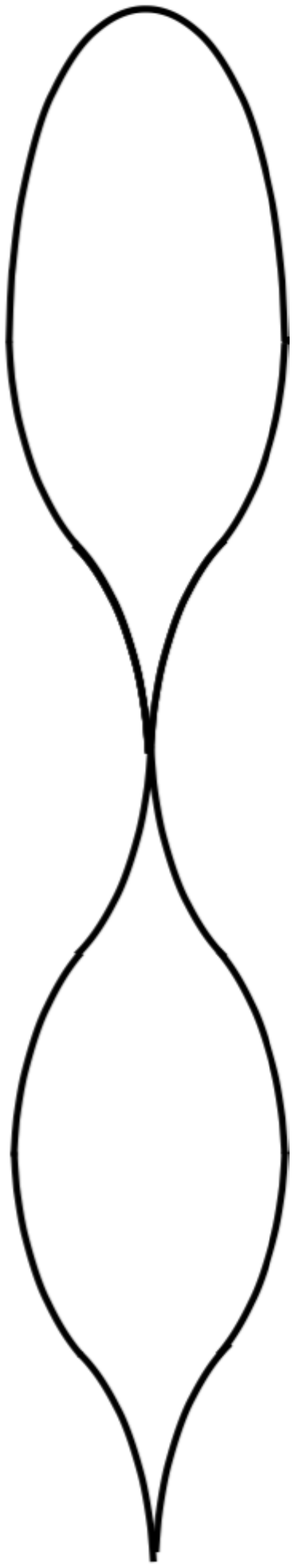}
\end{center}
\caption{Self touching curve, disconnecting the limit}\label{fig6c}
\end{figure}

We shall prove that $\g$ can not have self intersection points, other that the type above. The key ingredients are the local representation of the curve as a graph and  inequality \eqref{bh06}.
We shall analyze the different contact types between  two pieces of $\g$. Since the curves are graphs on an interval $[-\frac{l}{\sqrt{2}}, \frac{l}{\sqrt{2}} ]$, and the representing functions are ordered, we shall look to the orientation of each piece of curve.

\medskip
\noindent {\bf Case 1. Opposite orientation, not disconnecting.}
\begin{figure}[ht]
\begin{center}
\includegraphics[width=8cm]{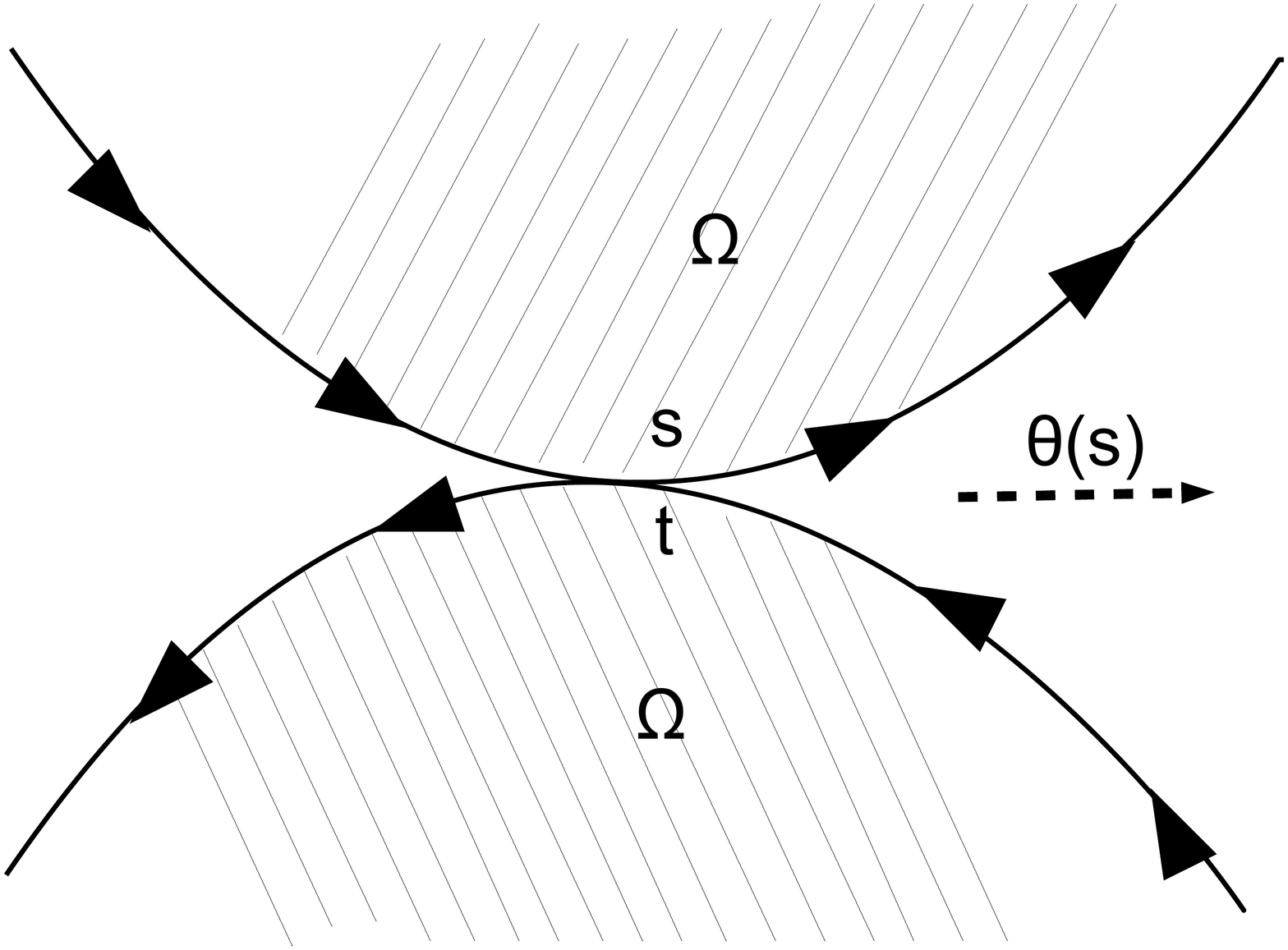}
\end{center}
\caption{Case 1: opposite orientation, not disconnecting.}
\end{figure}
Two branches of $\g$ touching at some point $\g(s)=\g(t)$, are represented as graphs of the functions $g_s,g_t$, on $[-\frac{l}{\sqrt{2}}, \frac{l}{\sqrt{2}} ]$. We assume that $g_s(0)=\g(s)=\g(t)=g_t(0)$ and  choose the couple $(s,t)$ such that for some $\vps >0$ we have
$$\forall u \in (0, \vps)\;\; g_s(u) >g_t(u),$$
otherwise we change the contact point. This inequality would imply the existence of points $s'>s$ and $t'<t$ such that $\theta (t') < \theta (s')-\pi$, which is in contradiction with \eqref{bh06}, so that this situation can not occur.

\medskip

\noindent {\bf Case 2. Contact of two branches of the same orientation.}
\begin{figure}[ht]
\begin{center}
\includegraphics[width=8cm]{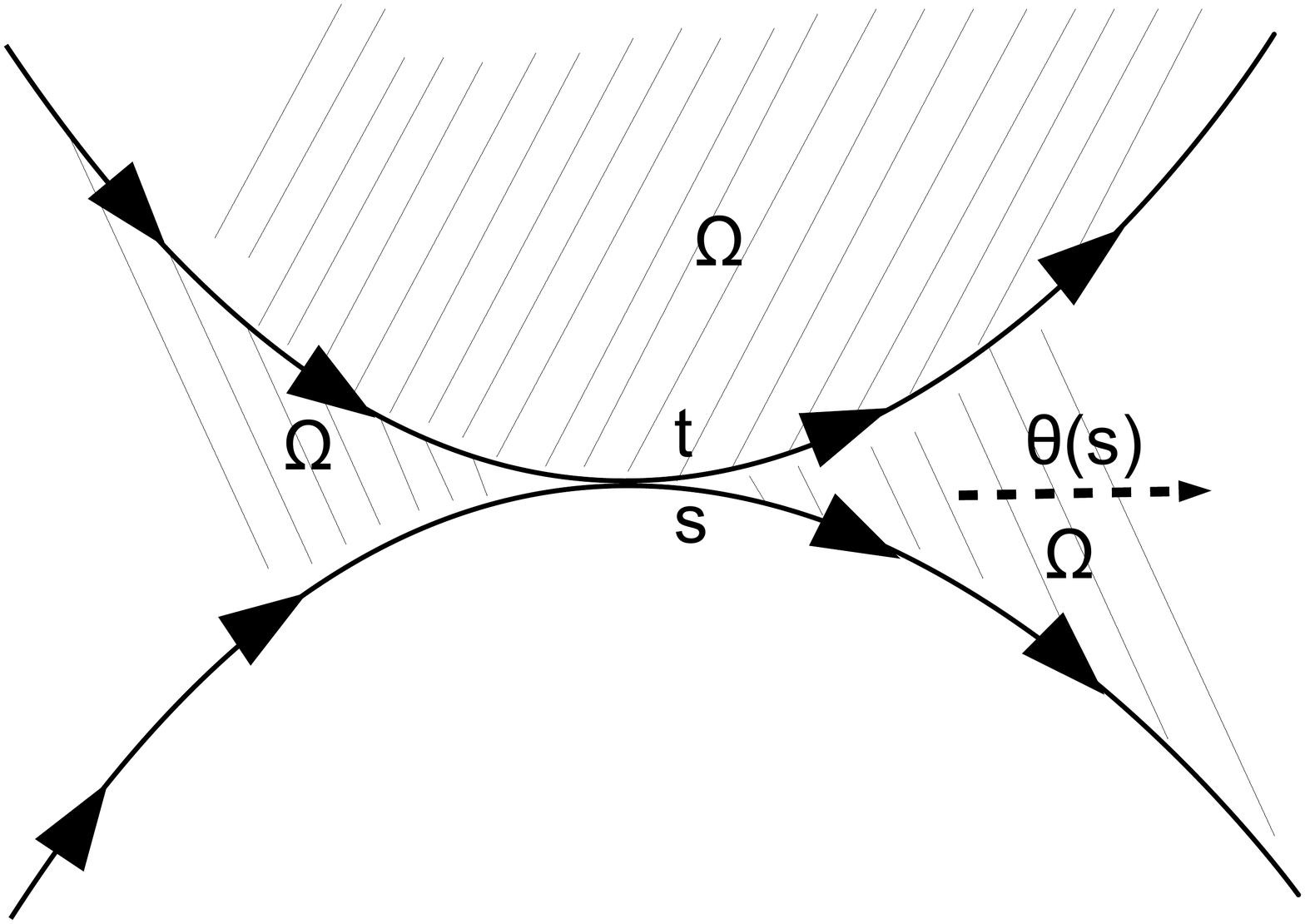}
\end{center}
\caption{Case 2: same orientation}
\end{figure}
From the simple connectedness, this situation implies that the touching point $\g(s)$ belongs to at least three branches, in particular between the graphs of  $g_s$ and $g_t$, there is a graph corresponding to piece of curve with opposite orientation. There are two possibilities: either this new contact corresponds to a point $t'\in (t,L)$ or to $s'\in (0,s)$. The first situation is in fact the case 1 between the contact points $s$ and $t'$. The second situation  leads also  to Case 1, but for the contact points $s'$ and $t'$,  so we conclude that the second case can not hold.

\medskip
\noindent {\bf Case 3. Opposite orientation,  disconnecting.}
\begin{figure}[ht]
\begin{center}
\includegraphics[width=8cm]{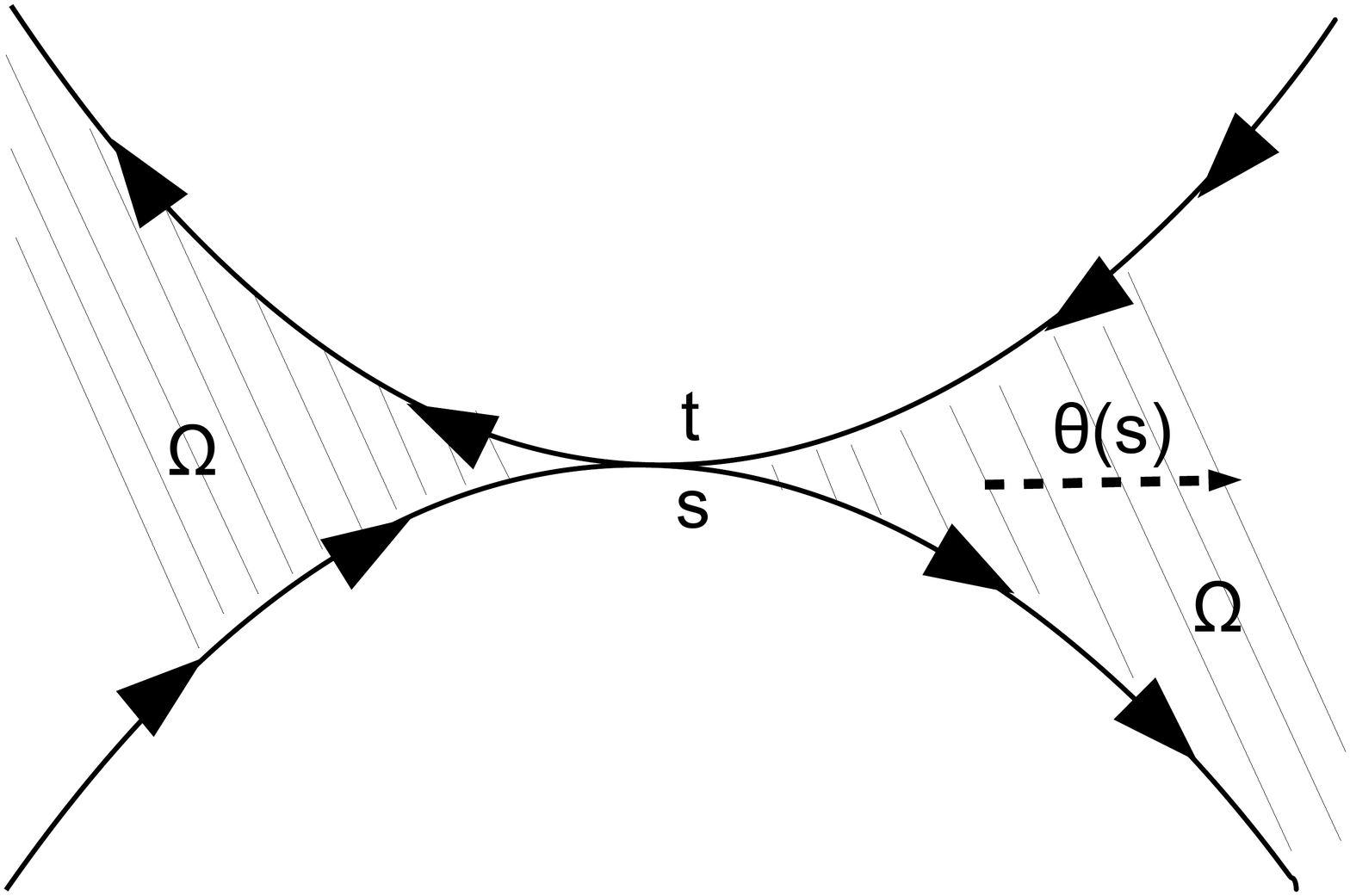}
\end{center}
\caption{Case 3: simple touch, disconnecting}
\end{figure}
This is the only remaining possibility for self-intersections. There may be several contact points, but every contact point is simple, otherwise we would fall in Case 2. So let us denote $\{(s_\alpha, t_\alpha)\}_\alpha$ the couple of parameters corresponding to the contact points. Because of the simple connectedness and of the absence of contact poins as in cases 1 and 2, we have that if $s_\alpha <s_\beta$ then
$t_\beta < t_\alpha$. Consequently, we can identify the contact point $(s^*,t^*)$ such that between  $s^*$ and $t^*$ there is no other contact, by setting $s^*=\sup_\alpha s_\alpha$ and $t^*=\inf_\alpha t_\alpha$. Of course, $s^* $ and $t^*$ can not collapse. Indeed,  in view of Lemma \ref{bh02} applied to $\g|_{[s_\alpha,t_\alpha]}$ if collapse occurs then the elastic energy would blow up. So $\g|_{[s^*,t^*]}$ is a Jordan curve for which all the area enclosed is part of $\Omega$, otherwise, because of the simple connectedness, a branch of the curve must pass through the contact point, bringing it to the case 2.

So $\g|_{[s^*,t^*]}$ is a drop, with lower elastic energy than $\g$ and enclosing a surface less than or equal to $|\Omega|$. This means that $\g|_{[s^*,t^*]}$ is a solution for problem \eqref{bh01}.

\end{proof}
\begin{lem}\label{bhe02}
There exists $R_0$, such that if the radius $R$ of the ball $B_R$ in Theorem \ref{bh03} satisfies $R \ge R_0$, then there exists a translation of the optimal drop which does not touch the boundary of $B_R$.
\end{lem}
\begin{proof}
The proof relies on Lemma \ref{be03}. Let $R\ge R_0$ (the value of $R_0$ will be precised at the end of the proof). Assume that $(\Om^*, \g^*)$ is an optimal drop for problem \eqref{bh01} which touches the boundary, such that there is no translation moving the drop at positive distance from the boundary. This means that the touching points between $\g^*$ and $B_R$ are distributed in such a way that they do not fit in an arc of length less than $\pi R$. Using Lemma \ref{be03}, if $R$ is large, e.g. $R \ge 300$, then the center of $B_R$ has to be inside $\Om^*$ together with a disc of radius $150$. Indeed, the longest piece of curve between two contact points or between a contact point and the singularity has a length less than $146$.  Consequently, the energy of  $(\Om^*, \g^*)$ is larger than the area of the disc of radius
$150$: $\pi\cdot 150^2$, in contradiction with its optimality.
\begin{figure}[ht]
\begin{center}
\includegraphics[width=8cm]{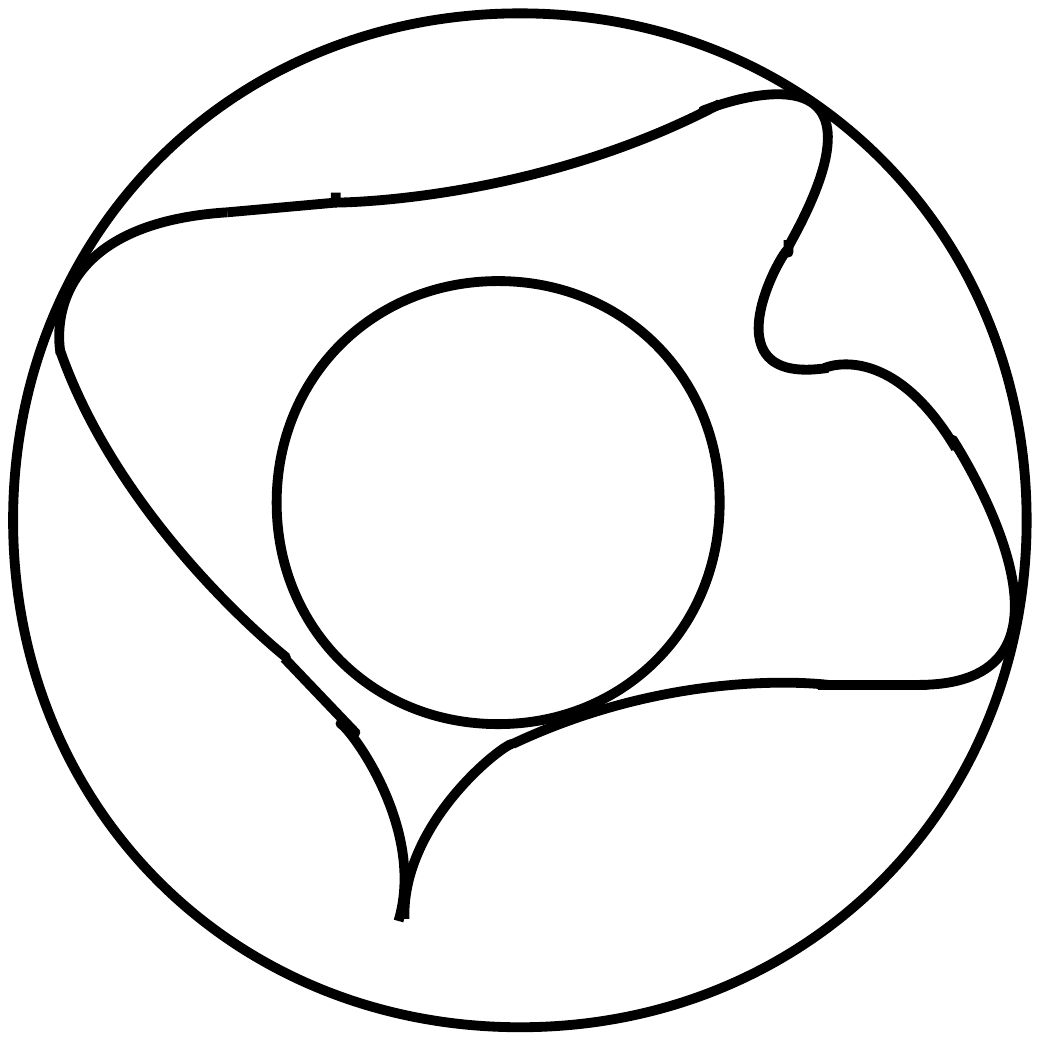}
\end{center}
\caption{An optimal drop touching the boundary}
\end{figure}
Taking $R_0= 300$, the lemma is proved.
\end{proof}

\begin{thm}\label{bhe01}
There exists a unique optimal drop $(\Om^*,\g^*)$ which minimizes the energy $E(\g)+A(\Omega)$ among all drops in $\R^2$.
This one is fully characterized by the optimality conditions (B1)-(B2) with a unique constant $C$
which can be determined.\\
Moreover
$$ E(\g^*)+A(\Om^*)  > \pi > 3\pi 2^{-\frac 52}= \frac 12 [E(\partial B_{2^{-1/3}})+ A(B_{2^{-1/3}})].$$
\end{thm}
\begin{rem}
Figure \ref{the optimal drop} gives the representation of the optimal drop.
\begin{figure}[ht]
\begin{center}
\includegraphics[width=8cm]{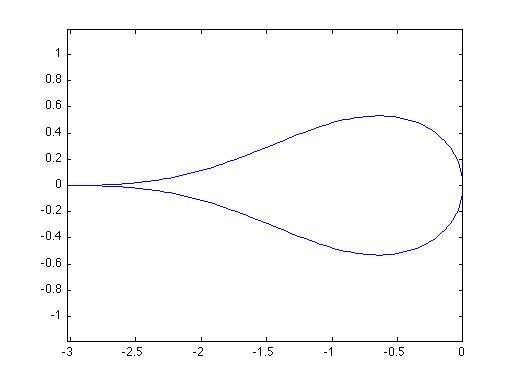}
\end{center}
\caption{The optimal drop}\label{the optimal drop}
\end{figure}
\end{rem}
\begin{proof}
The proof of existence follows from Theorem \ref{bh03} and Lemma \ref{bhe02}. The optimality conditions (B1)-(B4)
can be written on the whole $\g$ (except at the singularity) according to Theorem \ref{opti}. We start for $s=0$
at the origin which is the singular point with an horizontal tangent ($\theta(0)=0$). By (B4) and starshaped property, the point $Q$
is necessarily on the $x$-axis, the curvature $k(s)$ is negative for $s>0$ small and $k(s)\to 0$ when $s\to 0$.
The function $k(s)$ is periodic but we will prove below (see the end of the proof) that we have only one period for
the optimal drop and the curve is symmetric around the $x$-axis.
Therefore to characterize the optimal drop, we can proceed in the following way: for any constant $C>0$, we solve the ODE
\begin{equation}\label{drop1}
    \left\lbrace \begin{array}{c}
                   k''=-\frac 12 k^3 + 1 \\
                   k(0)=0 \\
                   k'(0)= - \sqrt{2C}
                 \end{array}\right.
\end{equation}
which has a unique solution. Let us denote by $s_M$ the value where $k$ is maximum with $k(s_M)=k_M$
(respectively $s_m$ and $k_m=k(s_m)$ for the minimum). The point $M_M$ of abscissa $s_M$
is necessarily on the $x$-axis and its tangent is vertical.
Thus, we look for the value of $C$ for which $\theta(s_M)=\int_0^{s_M} k(s) ds=\pi/2$.

\smallskip
We claim that conversely, if we find a value of $C$ for which $\int_0^{s_M} k(s) ds=\pi/2$, then we have found the
optimal drop. Indeed, since it satisfies the optimality conditions,
it suffices to check that the curve we obtain by $x(s)=\int_0^s \cos\theta(t) dt$ and
$y(s)=\int_0^s \cos\theta(t) dt$ with $\theta(s)=\int_0^s k(t) dt$ is an admissible drop. Since
$M_M$ is the point where the curvature is maximum, according to (B3), it is the point on $\g$
which is the farthest to $Q$. But since the tangent is vertical at this point it is necessarily on the $x$-axis: $y(s_M)=0$
and the total length of the curve is $2s_M$. Now, since $k$ is symmetric with respect to $s_M$ (see (ODE1) in the Appendix),
$k(s_M+t)=k(s_M-t)$ which provides after integration: $\theta(s_M+t)=\pi - \theta(s_M-t)$. This identity gives $\theta(2s_M)=\pi$
and
\begin{eqnarray*}
x(2s_M)=\int_0^{s_M} \cos\theta(t) dt+\int_{s_M}^{2s_M} \cos\theta(t) dt=\int_0^{s_M}
\cos\theta(t)+\cos(\pi-\theta(t)) dt=0\\ \nonumber
y(2s_M)=\int_0^{s_M} \sin\theta(t) dt+\int_{s_M}^{2s_M} \sin\theta(t) dt=\int_0^{s_M}
\sin\theta(t)+\sin(\pi-\theta(t)) dt=2 y(s_M) =0 \\ \nonumber
\end{eqnarray*}
which shows that the curve $\g$ is a drop.

Thus to prove uniqueness of the optimal drop, we need to prove that we can find only one $C>0$ for which $I(C):=\int_0^{s_M} k(s) ds=\pi/2$.
Let us write
$$\int_0^{s_M} k(s) ds=\int_0^{2s_m} k(s) ds + \int_{2s_m}^{s_M} k(s) ds=2\int_0^{s_m} k(s) ds + \int_{2s_m}^{s_M} k(s) ds$$
where we used the symmetry of $k$ with respect to $s_m$, see (ODE1). This symmetry also shows that $k(2s_m)=0$.
We are going to prove uniqueness of $C$ (and therefore of the optimal drop) by proving that the function $C\mapsto \int_0^{s_M} k(s) ds$
is strictly decreasing.
Let us perform the change of variable $u=k(s)$ in each above integral. It comes, using (B2) to express $k'$:
\begin{eqnarray}
  \int_{2s_m}^{s_M} k(s) ds = \int_0^{k_M} \frac{u}{\sqrt{2C+2u-u^4/4}} \,du\\ \nonumber
  \int_{0}^{s_m} k(s) ds = -\int_0^{k_m} \frac{u}{\sqrt{2C+2u-u^4/4}} \,du\\ \nonumber
\end{eqnarray}
Now to compute the derivative of the first integral $I_1(C)$ with respect to $C$, we make the change of variable $u=k_Mx$, it comes
$$I_1(C)=\int_0^{1} \frac{k_M^2 x}{\sqrt{2C+2k_M x-k_M^4 x^4/4}} \,dx$$
We compute the derivative of $I_1$ using $\frac{d k_M}{d C}=2/(k_M^3 -2)$ (see (ODE3) in the appendix) and an easy computation gives
$$\frac{d I_1}{d C}= \int_0^{1} \frac{6 k_M^2 x(x-1)}{(k_M^3-2)\left(2C+2k_M x-k_M^4 x^4/4\right)^{3/2}} \,dx$$
which is clearly negative. In the same way, we get for the second integral $I_2(C)=\int_{0}^{s_m} k(s) ds$:
$$\frac{d I_2}{d C}= -\int_0^{1} \frac{6 k_m^2 x(x-1)}{(k_m^3-2)\left(2C+2k_m x-k_m^4 x^4/4\right)^{3/2}} \,dx$$
which is also negative, proving the uniqueness of a solution $C$ for the equation $I_1(C)+2I_2(C)=\pi/2$.
Let us remark that a simple computation yields $I(0)=\frac{2\pi}{3}$ while the limit of $I(C)$ when $C$ goes
to $+\infty$ is $-\frac{\pi}{2}$ confirming that there exists a solution to our problem.

\medskip
Let us estimate from below the energy of the optimal drop.
Denote by $s_1=2 s_m$ the first positive zero of $k$, we recall that $s_m$ is the first minimum of $k$ and $k_m= k(s_m)$, $s_M$ the first maximum of $k$ and $k_M= k(s_M)$. From (B2) $k_m$ and $k_M$ are the real roots of the polynomial (which is concave)
\begin{equation}\label{bhe09}
P_C(X)= -\frac 14 X^4 + 2 X+2C.
\end{equation}
The maximum of $P_C$ is at $X= 2^{\frac 13}$  and $P_C(0)= 2C$. We have
\begin{equation}\label{bhe10}
k_m<0\le 2^{\frac 13} \le k_M
\end{equation}
($k_m$ can not be nonnegative, otherwise the set $\Om^*$ would be convex).

Moreover, when $C$ increases, $k_M(C) $ is increasing while $k_m(C)$ is decreasing (with increasing absolute value $|k_m(C)|$), because we translate the curve $y= - \frac 14 x^4 +2x$ up).

If we denote $S=k_M+k_m$ and $P=k_mk_M$ the sum and the product of those two roots, classical elimination and relation between roots provide
\begin{equation}\label{bhe11}
S^2=P-\frac{8C}{P} \qquad -\frac 8S = P + \frac{8C}{P}
\end{equation}
while the two complex roots $z_0, \ov z_0$ satisfy $z_0+\ov z_0 =-S$, $ z_0\ov z_0 = -\frac{8C}{P}$.

Since $P \le 0$ and $C>0$, the last equation gives $S>0$. Let us come back to the computation of the elastic energy of the optimal drop $(\Om^*, \g^*)$
$$E(\g^*)= \int_0^{s_M} k^2 ds.$$
Now
$ \int_0^{s_M} k^2 ds \ge  \int_{s_m}^{s_M} k^2 ds$ and $k$ is increasing from $s_m$ to $s_M$ (since $k'$ can only vanish at zeroes of  $P_C(X)$, which only correspond to maxima $k_M$ and minima $k_m$). We perform the change of variable $x=k(s)$ on this interval
$dx= k'(s) ds=\sqrt{2C+2k -\frac 14 k^4}ds$. Therefore
$$E(\g^*)\ge \int_{s_m}^{s_M} k^2 ds = \int_{k_m}^{k_M} \frac{x^2}{\sqrt{2C+2x -\frac 14 x^4}} dx.$$
We want to find a lower bound of this integral. For this purpose, we write (following \eqref{bhe09})
$$P_C(x)= \frac 14(k_M-x)(x-k_m)(x^2+Sx - \frac{8C}{P}).$$
Now, the parabola $y= \frac 14 (x^2+Sx-\frac{8C}{P})$ is symmetric with respect to $-\frac  S2$, an since $     \frac{k_M+k_m}{2} = \frac  S2 \ge - \frac  S2$, the maximum of $y$ on the interval $[k_m, k_M]$ is equal to
\begin{equation}\label{bhe12}
F^2 = \frac 14 ( k_M^2 + S k_M- \frac{8C}{P})= \frac 14 (2k_M^2 +k_mk_M - \frac{8C}{P}) = \frac 14 (3k_M^2+2k_mk_M+k_m^2),
\end{equation}
where we have used \eqref{bhe11} for the last equality.
Thus
$$\int_{k_m}^{k_M} \frac{x^2}{\sqrt{2C+2x-\frac 14 x^4}}dx \ge \frac 1F \int_{k_m}^{k_M} \frac{x^2}{\sqrt{(k_M-x)(x-k_m)}}dx.$$
This last integral can be computed explicitly and gives
\begin{equation}\label{bhe13}
E(\g^*)\ge \frac 1F \frac{3k_M^2+2k_mk_M+3k_m^2   }{4}\frac {\pi}{2}.
\end{equation}
We have
 $F \le \frac 12 \sqrt{3k_M^2+2k_mk_M+3k_m^2 } $ and \eqref{bhe13} gives
 \begin{equation}\label{bhe14}
 E(\g^*)\ge \frac{\pi}{4} \sqrt{3k_M^2+2k_mk_M+3k_m^2 }.
 \end{equation}
It remains to get a bound for the quantity $ H= 3k_M^2+2k_mk_M+3k_m^2 $ which depends only on $C$. We discuss two cases.

\noindent {\bf Case A.} If $C \ge 1$, $H= k_M^2 +2 k_M(k_M+k_m) + 3 k_m^2\ge k_M^2 + 3 k_m^2$. Both mappings
$C \mapsto k^2_m$, $C\mapsto k_M^2$ are increasing, thus $C \ge k_M^2 (1) + 3 k^2_m(1)$. We study $P_1(X)=-\frac 14 X^4 +2X+2$
$$P_1(\frac 73)= -\frac{241}{324}\mbox {    and    } P_1(\frac 94)= \frac{95}{1024},$$
we get
\begin{equation}\label{bhe14.b}
\frac 94 \le k_M(1) \le \frac 73.
\end{equation}
While from $P_1(-1)= -\frac 14$ and $P_1( -\frac{9}{10}) = \frac{1439}{40000}$, we get
\begin{equation}\label{bhe15}
-1 \le k_m(1)\le - \frac{9}{10}.
\end{equation}
It follows that $H \ge \Big ( \frac{9}{4}  \Big )^2+ 3 ( \frac{9}{10}  \Big )^2 = \frac{2997}{400}\approx 7.4925.$

\noindent {\bf Case B.} In the case $0 \le C\le 1$, we use $k_M^2(C) \ge k_M^2(0)=4$, $k^2_m(C)\ge 0$ and $|k_M(C)k_m(C)| \le |k_M(1)k_m(1)|\le \frac 73$ to get
$$H= 3k_M^2+2k_mk_M+3k_m^2\ge 12-\frac{14}{3}=\frac{22}{3}=7.333...$$
So in any case, $H \ge \frac{22}{3}$. It follows from \eqref{bhe13} that
\begin{equation}\label{bhe16}
E(\g^*)\ge \frac{\pi}{4} \sqrt{\frac{22}{3}}.
\end{equation}
Now, integrating (B4) on the curve, we get $2 A(\Om^*)=\int_{\g^*} \overrightarrow{QM}\cdot \vec{ \nu} ds = \frac 12 \int_{\g^*}k^2 ds = E(\g^*)$.

Therefore
\begin{equation}\label{bhe17}
E(\g^*) + A(\Om^*) = \frac 32 E(\g^*) \ge \frac{3\pi}{8} \sqrt{\frac{22}{3}} > \pi > 3\pi 2^{-\frac 53}.
\end{equation}
Let us now conclude by proving that the optimal drop has only one period of the function $k(s)$.
The estimate (\ref{bhe16}) we get is actually true on any possible period. Therefore, if we have a solution $(\g_2^*,\Om_2^*)$
with at least two periods,
we would have $E(\g_2^*)\ge \frac{\pi}{2} \sqrt{\frac{22}{3}}$, therefore like in (\ref{bhe17}) its total
energy would satisfy $E(\g_2^*) + A(\Om_2^*) > 2\pi$.
Now, proceeding in a similar way as we did for the estimate from below, we can get (details omitted) an estimate from above
for an optimal drop with only one period which is
$$E(\g^*) + A(\Om^*) \leq  2\pi$$
(the exact value is $E(\g^*) + A(\Om^*) \simeq 4.6823$) therefore, any critical point with more than one period cannot be optimal.

\end{proof}

\section{Proof of Theorem \ref{bhe05}}\label{s4ab}
 With the notation settled in Sections \ref{prel} and \ref{drops} we return to problem \eqref{bhe06}, and write
\begin{equation}\label{bh01l}
\inf\{  E(\g) +|\Om| : \Om \mbox{ smooth, bounded, simply connected set }, \partial \Om =\g\}.
\end{equation}

First of all we recall that among all circles, the optimal one has the radius $r= 2^{-\frac 13}$. Let us consider $R\ge 300$ and solve the problem
\begin{equation}\label{bh01l2}
\inf\{  E(\g) +|\Om| : \Om \mbox{ smooth, bounded, simply connected set}, \Om \sq B_R, \partial \Om =\g\}.
\end{equation}
 Using the same arguments as in Section \ref{drops}, a minimizing sequence will converge to a couple $(\Om, \g)$. Two possibilities occur. Assume first that there are self intersections. In this case the limiting couple   $(\Om, \g)$ contains at least two drops, as in Case 3 of Theorem \ref{bh03}. Following Theorem \ref{bhe01}, this configuration can not be optimal since the energy of $\Om$ is larger than the double of the optimal energy of a drop, so it is excluded.

The second situation is that $(\Om, \g)$ does not have self-intersections. Since the radius is large enough, for a suitable translation the loop does not touch the boundary of the ball, as in Lemma \ref{bhe02}.  Moreover, in this case the optimality conditions $OM.\nu = \frac 12 k^2$ can be written on the full boundary.

\begin{rem}
This condition recalls the result of Ben Andrews (see Theorem 1.5 in \cite{andrews}) which, under the hypothesis of positive curvature would allow directly to conclude that the curve is a circle with a radius  equal to $2^{-\frac 13}$ (by direct computation). As the curvature is not known to be positive, for our conclusion we shall use again the optimality conditions.
Actually, Andrews's result does not hold true for non convex curves. Indeed, Figure \ref{23fig} (which has been obtained using the optimality conditions) shows a curve which satisfies $OM.\nu = \frac 12 k^2$ on the whole boundary.
\end{rem}

If the curvature is not constant, we can assert that $\Om$ is star shaped and  the structure of $\g$ is a union of periods consisting of two branches $\g_1$ and $\g_2$, where $\g_1: [0, l]\ra \R^2$ is a branch of the curve with increasing curvature such that $\g(0)=k_m, \g(l)=k_M$ and $\g_2: [l, 2l]\ra \R^2$ is a congruent branch with decreasing curvature from $k_M$ to $k_m$. Following \eqref{B4} and (ODE4), $\g$ consists of  one, two or three periods $(\g_1,\g_2)$ (as explained in the proof of Theorem \ref{bhe01}.
From the optimality conditions $(B_1)-(B_4)$ one can eliminate any of those three configurations, since their energy is much larger than the one of the ball.
Indeed, in the case of two or three periods, a couple $(\g_1,\g_2)$ has a cap $\g_{(l-a, l+a)}$, where $a$ is chosen such that $\nu_{\g(a)}$ is orthogonal to the segment $O\g(l)$. As on Figure \ref{23fig}, we can cut and reflect along the line $\g(l-a), \g(l+a)$ to get a new domain
which smaller area and smaller elastic energy.
\begin{figure}[!h]
\begin{center}
\includegraphics[width=10cm,trim=50 200 50 200,clip=true]{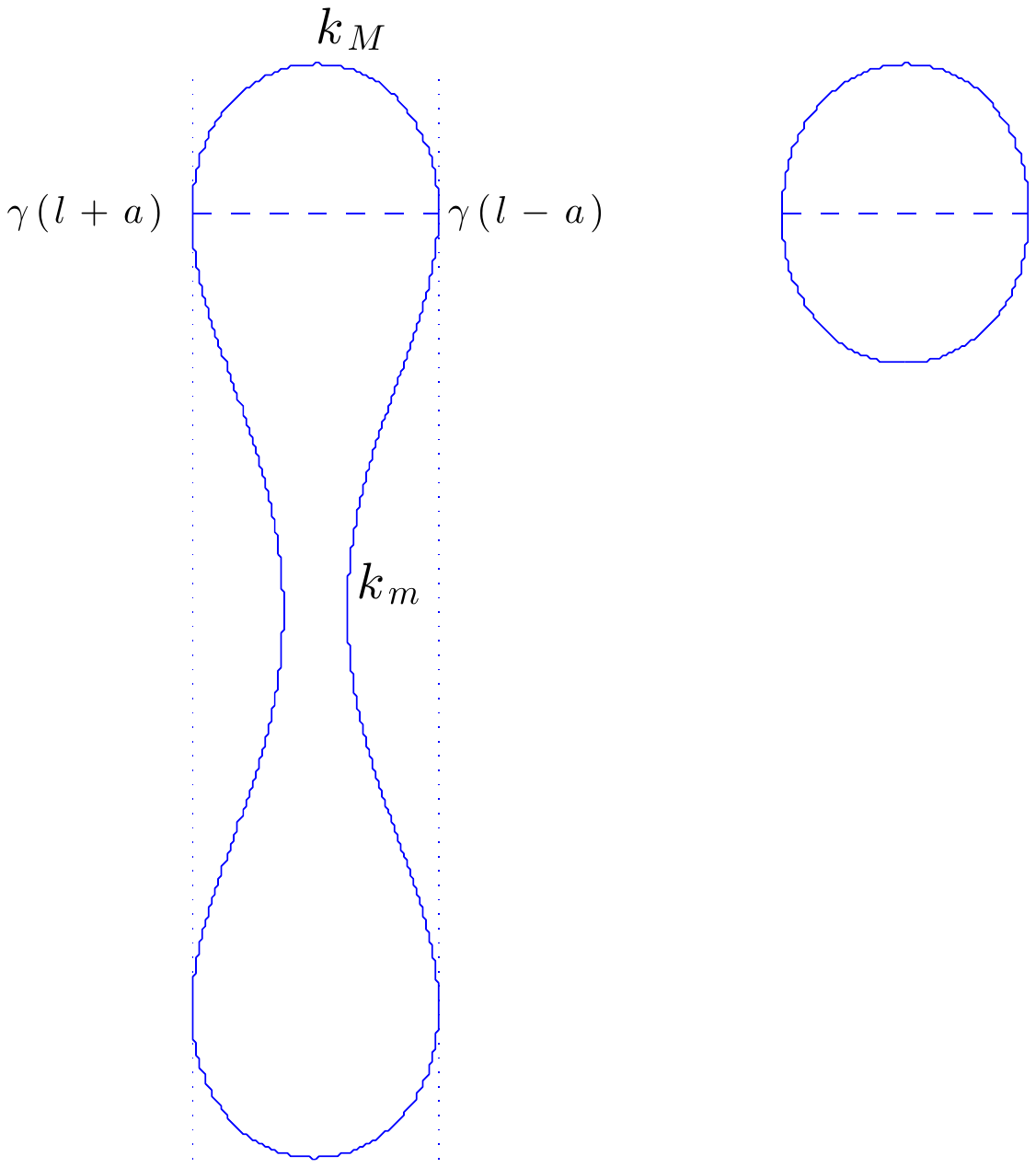}
\end{center}
\caption{The case of more than one period.}\label{23fig}
\end{figure}
If there is only one period, the argument is similar since we can center-symmetrize the branch from $\g(l-a)$ to $\g(l)$, where $a$ is chosen such that the normal at  $\g(l-a)$ is parallel to the segment $O\g(l)$  (see Figure \ref{1fig}).
\begin{figure}[!h]
\begin{center}
\includegraphics[width=10cm,trim=50 200 50 200,clip=true]{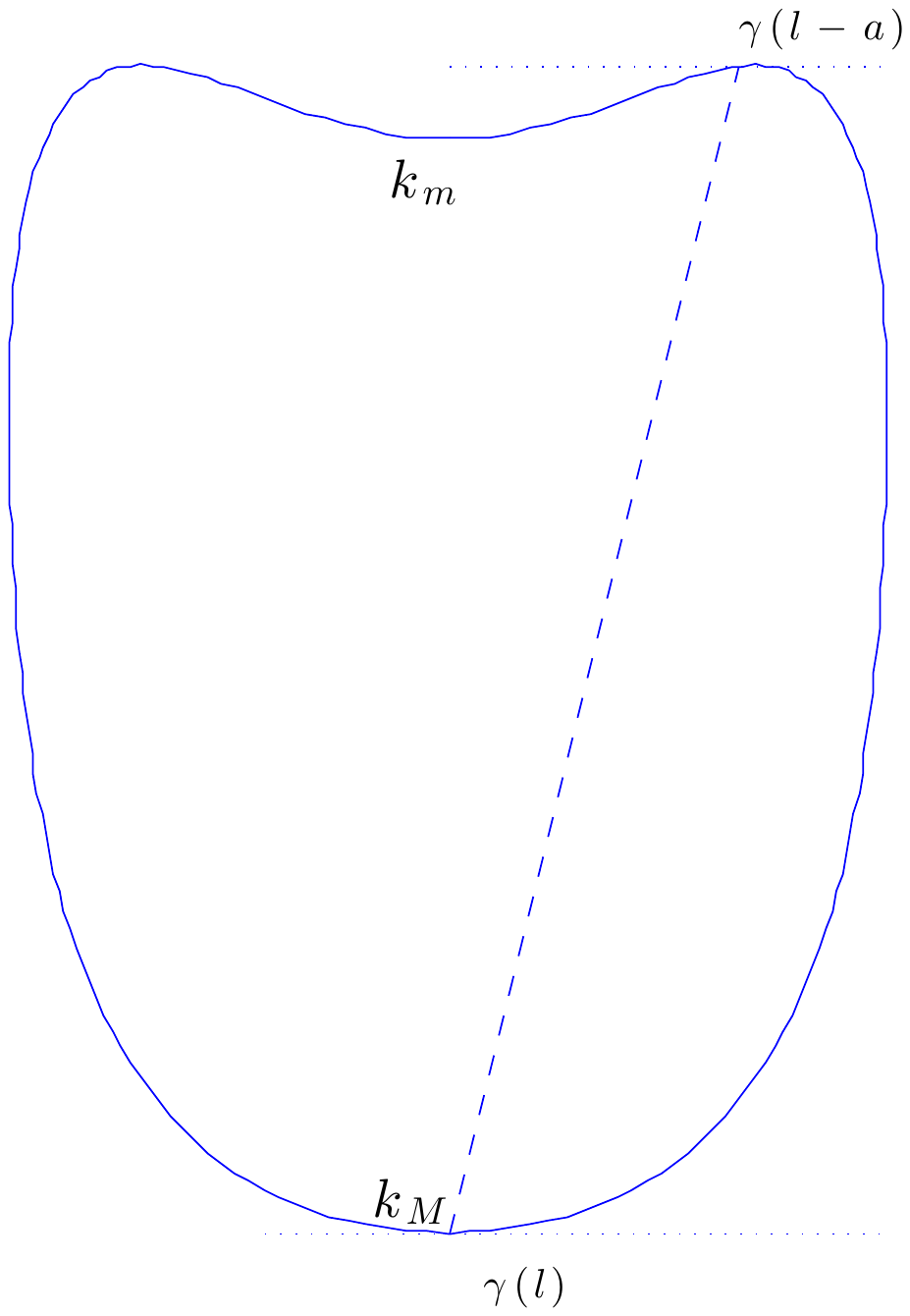}
\end{center}
\caption{The case of  one period.}\label{1fig}
\end{figure}

Both previous constructions are admissible as a consequence of the optimality conditions  $(B_1)-(B_4)$.

\section{Appendix: analysis of the ODE issued from optimality conditions}
In this section, we give several properties of the following ODE in nonstandard form
$${k'}^2= -\frac 14 k^4+2k +2C,$$
where $C\in \R$ is a constant. This ODE is issued from the optimality conditions on a free branch of a minimizer for our problem, see Theorem \ref{opti}. We also refer the reader to reference \cite{bht14} for related analysis.

Clearly, $C \ge -\frac 34 2^{ \frac13}\approx -0.944$, otherwise the right hand side is negative. We denote $k_m(C)\le k_M(C)$
the two real roots of the polynomial $P_C(X)= -\frac 14 X^4+2X +2C$, or if there is no ambiguity simply $k_m, k_M$.

Here we gather some immediate facts concerning this ODE.

\begin{enumerate}
\item[(ODE1)] The solution of the ODE is periodic (the period is denoted by $T$), symmetric with respect to its minimum or maximum.

\item[(ODE2)]\label{ODE2}
 The only local minima (maxima) are actually global minima (maxima, respectively) and correspond to $k=k_m$ ($k=k_M$, respectively), and $k$ is monotone between these two values.

 \item[(ODE3)]\label{ODE3} The mapping $C\mapsto k_M(C)$ is increasing and its range is from $2^{ \frac13}$ to $+\infty$, while the mapping
  $C\mapsto k_m(C)$ is decreasing and its range is from $-\infty$ to
 $2^{ \frac13}$.  Moreover, $k_m(C)<0$ when $C>0$, $\frac 94 \le k_M(1)\le \frac 73$, $-1 \le k_m(1) \le -\frac {9}{10}$, $-C \le k_m(C)$. As well,  $k_M(C) \ge 2+C$ for $-\frac 32 \times 2^{ \frac13} \le C\le 0$.

 \item[(ODE4)]  The integral $\frac 12 \int _0^T k^2 ds $ on one period is estimated from below
 $$\frac 12 \int _0^T k^2 ds \ge \frac {\pi}{ 4}  \sqrt{\frac{22}{ 3}}.$$
\end{enumerate}

\smallskip
\noindent The proof of (ODE1) is classical, either working with the closed orbit, or using an explicit form of the solution thanks to elliptic functions.

\smallskip
\noindent The proof of (ODE2) is easy since $k'$ can vanish only at the zeroes of $P_C$.

\smallskip
\noindent For the proof of (ODE3) we notice that $\frac {d k_M}{dC} =\frac{2}{k_M^3-2}> 0$ and $\frac {d k_m}{dC} =\frac{2}{k_m^3-2}< 0$, $k_m(0)=0, k_M(0)=2$, $P_C(-C)<0\Lra k_m(C)\ge -C$, $P_C(2+C)= -C[\frac 14 C^3 +2C^2+6C+4]\ge 0$ and the bounds for $k_m(1), k_M(1)$ have been obtained in \eqref{bhe14.b}, \eqref{bhe15}.

\smallskip
\noindent  The proof of (ODE4): we have already proved this inequality in Section \ref{drops}, when $C\ge 0$. It remains the case $-\frac 34 2^{\frac 13} \le C \le 0$.  In this case, we have $k_m \ge C$ and $k_M\ge 2+C$, so $3k_M^2+2k_mk_M+3k_m^2\ge 4C^2 +8C+2\ge 8\ge \frac {22}{3}$, and the result follows in the same way.

\subsection*{Acknowledgement
 and History}
This work has been initiated during a stay of the two authors in the Isaac Newton Institute, Cambridge in March 2014 during the programme
"Free Boundary Problems and Related Topics". The authors are very grateful to the Institute for the very good and stimulating
atmosphere here. The results of this paper have been announced in several Conferences (e.g. Petropolis in August 2014, Linz
in October 2014). While completing this manuscript, we learnt from  Bernd Kawohl that he, with Vincenzo Ferone and Carlo Nitsch,
proved our Theorem \ref{bhe05} by a different way. Their manuscript is available on ArXiv, see Reference \cite{FKN}.

\noindent Dorin Bucur, Laboratoire de Math\'ematiques, CNRS UMR 5127,
Universit\' e de Savoie, Campus Scientifique, 73376 Le-Bourget-Du-Lac, France,  {\tt{dorin.bucur@univ-savoie.fr}}

\medskip
\noindent Antoine Henrot, Institut Elie Cartan de Lorraine, CNRS UMR 7502 and Universit\'e de Lorraine, BP 70239
54506 Vandoeuvre-l\`es-Nancy, France,
\tt {antoine.henrot@univ-lorraine.fr}

\end{document}